\documentclass[12pt,reqno]{article}

\usepackage{amsmath,amssymb,amsthm}

\usepackage{pictexwd}

\DeclareMathOperator{\Ext}{Ext}
\DeclareMathOperator{\End}{End}
\DeclareMathOperator{\Supp}{Supp}
\DeclareMathOperator{\mo}{mod}
\DeclareMathOperator{\Hom}{Hom}

\newcommand{\seqnum}[1]{\underline{#1}}

\theoremstyle{plain}
\newtheorem{theorem}{Theorem}
\newtheorem{corollary}[theorem]{Corollary}
\newtheorem{lemma}[theorem]{Lemma}
\newtheorem{proposition}[theorem]{Proposition}

\theoremstyle{remark}
\newtheorem{tri}{\bf }[section]
\newtheorem{triS}{\bf S}[section]
\newtheorem*{no-text}{}

\newtheorem{remark}{Remark}

\newcommand{\ssize}[1]{\smallmatrix #1\endsmallmatrix}
\newcommand{\arr}[2]{\arrow <1.5mm> [0.25,0.75] from #1 to #2}

\begin{document}
\title{The Numbers of Support-Tilting Modules\\ for a Dynkin Algebra}
\author{
Mustafa A. A. Obaid\\ \texttt{drmobaid$@$yahoo.com}
\\ \\
S. Khalid Nauman\\ \texttt{snauman$@$kau.edu.sa}
\\ \\
Wafaa M. Fakieh\\ \texttt{wafaa.fakieh$@$hotmail.com}
\\ \\
Claus Michael Ringel\\  \texttt{ringel$@$math.uni-bielefeld.de}
\\ \\
King Abdulaziz University, P O Box 80200\\ Jeddah, Saudi Arabia\\}
\date{}
\maketitle

\begin{abstract} 
The Dynkin algebras are the
hereditary artin algebras of finite representation type. The paper exhibits the
number of support-tilting modules for any Dynkin algebra.
Since the support-tilting modules for a Dynkin algebra of Dynkin type $\Delta$
correspond bijectively to the generalized non-crossing partitions
of type $\Delta$, the calculations presented here may also be considered as a
categorification of results concerning the generalized non-crossing partitions.
In the Dynkin case $\mathbb A$, we obtain the Catalan triangle,
in the cases $\mathbb B$ and $\mathbb C$ the increasing part of the
Pascal triangle, and finally in the case $\mathbb D$ an 
expansion of the increasing part of
the Lucas triangle.

\end{abstract}

\section{Introduction}

Let $\Lambda$ be a  hereditary
artin algebra. Here we consider left $\Lambda$-modules of finite
length and call them just modules. The category of all modules will be denoted by $\mo\Lambda$.
We let $n = n(\Lambda)$ denote the {\it rank} of $\Lambda$; this is by definition
the number of simple modules (always, when counting numbers of modules
of a certain kind, we actually mean the number of isomorphism classes).
Following earlier considerations of Brenner and Butler,
tilting modules were defined in \cite{[HR1]}. In the present setting, a {\it tilting module}
$M$ is a module without self-extensions with precisely $n$ isomorphism classes of
indecomposable direct summands, and we will assume, in addition, that $M$ is multiplicity-free.
The endomorphism ring of a tilting module is said to be
a tilted algebra. There is a wealth of papers devoted to tilted algebras, and the 
{\it Handbook of
Tilting Theory} \cite{[AHK]} can be consulted for references.

The present paper deals with the Dynkin algebras: these are the connected hereditary artin algebras
which are representation-finite, thus their valued quivers are of Dynkin type
$\Delta_n = \mathbb A_n, \mathbb B_n, \dots, \mathbb G_2$ (see \cite{[DR1]}). Its aim is to discuss the number
of tilting modules for such an algebra.
The corresponding tilted algebras were classified
by various authors in the eighties. It seems to be clear that a first step of such a classification
result was the determination of all tilting modules, however there are only few traces in the literature (also the Handbook \cite{[AHK]} is of no help).
Apparently, the relevance of the number of tilting modules 
was seen at that time only in special cases.
The tilting modules for a linearly ordered quiver of type $\mathbb A_n$
were exhibited in \cite{[HR2]} and Gabriel \cite{[G]} pointed out that here
we encounter one of the numerous appearances of the Catalan numbers
$\frac1{n+1}\binom{2n}n$. For the cases $\mathbb D_n$, the number
of tilting modules was determined by Bretscher-L\"aser-Riedtmann \cite{[BLR]} in their
study of self-injective representation-finite algebras.

Given a module $M$, we let $\Lambda(M)$ denote its
{\it support algebra;} this is the factor algebra of $\Lambda$
modulo the ideal which is generated by all idempotents $e$ with $eM = 0$ and is again a hereditary
artin algebra (but usually not connected, even if $\Lambda$ is connected). The rank of the support algebra
of $M$ will be called the {\it support-rank} of $M$. A module $T$ is said to be
{\it support-tilting} provided $M$ considered as a $\Lambda(M)$-module is a tilting module.
It may be well-known that the number of tilting modules of a Dynkin algebra depends only on its
Dynkin type; at least for path algebras of quivers we can refer to Ladkani \cite{[L]}. Section \ref{algebras} 
of the present paper provides a proof in general.
It follows that the number of support-tilting modules 
with support-rank $s$ also depends only on the type $\Delta_n$; we let 
$a_s(\Delta_n)$ denote the number of support-tilting $\Lambda$-modules 
with support-rank $s$, where $\Lambda$ is of type $\Delta_n$.
 Of course, $a_n(\Delta_n)$ is just the number of tilting modules, 
and we denote by $a(\Delta_n)$ the
number of all support-tilting modules; thus $a(\Delta_n) = \sum_{s=0}^n a_s(\Delta_n)$.

The present paper presents the numbers $a(\Delta_n)$ and $a_s(\Delta_n)$ for $0\le s \le n$ in a unified way. 
Of course, the exceptional cases $\mathbb E_6, \mathbb E_7,\mathbb E_8, \mathbb F_4, \mathbb G_2$
can be treated with a computer (but actually, also by hand); thus our main interest lies in the series
$\mathbb A, \mathbb B, \mathbb C, \mathbb D$.
In the case $\mathbb A$, we obtain in this way the {\bf Catalan triangle} \seqnum{A009766},
in the case $\mathbb B$ and $\mathbb C$ the increasing part of the
{\bf Pascal triangle,} and finally in the case $\mathbb D$ an 
{\bf expansion} of the increasing part of
the {\bf Lucas triangle} (see Section \ref{triangles}; an outline will be given later in the introduction).

\subsection{\bf The numbers}
    All the numbers which are presented here for the cases $\mathbb A, \mathbb B, \mathbb C, \mathbb D$ are related to the binomial
    coefficients $\binom st$ and they coincide for $\mathbb B_n$ and $\mathbb C_n$ (as we will show in
Section \ref{algebras}); thus it is sufficient to deal with the cases $\mathbb A, \mathbb B, \mathbb D$.
For $\mathbb B$,
the binomial coefficients themselves will play a dominant role. For the cases $\mathbb A$ and $\mathbb D$,
suitable multiples are relevant. In case $\mathbb A$, 
these are the Catalan numbers $C_n = \frac1{n+1}\binom{2n}n$, as well as related numbers.
For the case $\mathbb D$, it will be convenient to use the notation
$\left[\smallmatrix t\cr s \endsmallmatrix\right] = \frac{s+t}t\binom t s$ as proposed by
Bailey \cite{[B]}, since the
relevant numbers in case $\mathbb D$ can be written in this way. 

Hubery and Krause \cite{[HK]} have pointed out that
the numbers $a(\Delta)$ for the simply laced diagrams $\Delta$ 
were discussed already in 1987 by Gabriel and de la Pe\~na
\cite{[GP]}, but let us quote ``although 
they have the correct number for $\mathbb E_8$, their numbers for $\mathbb E_6$ and $\mathbb E_7$ are
slightly wrong''.
	 
\begin{theorem}	  
The numbers $a(\Delta_n)$ and $a_s(\Delta_n)$ for $0\le s \le n$:
$$
\hbox{\beginpicture
	\setcoordinatesystem units <1.8cm,1.2cm>
	\plot -.3 -.5 7.25 -.5 /
	\plot 0.4 0.3  0.4 -3.4 /
	\put{$\Delta_n$} at 0 0
	\put{$\mathbb A_n$} at 1 0
	\put{$\mathbb B_n, \mathbb C_n$} at 2 0
	\put{$\mathbb D_n$} at 3 0
	\put{$\mathbb E_6$} at 4 0
	\put{$\mathbb E_7$} at 4.8 0
	\put{$\mathbb E_8$} at 5.6 0
	\put{$\mathbb F_4$} at 6.4 0
	\put{$\mathbb G_2$} at 7 0

	\put{$a_n(\Delta_n)$} at 0 -1
	\put{$\frac1{n+1}\binom{2n}n$} at 1 -1
	\put{$\binom{2n-1}{n-1}$} at 2 -1
	\put{$\left[{\smallmatrix 2n-2\cr n-2\endsmallmatrix}\right]$} at 3 -1
	\put{$418$} at 4 -1
	\put{$2\,431$} at 4.8 -1
	\put{$17\,342$} at 5.6 -1
	\put{$66$} at 6.4 -1
	\put{$5$} at 7 -1

	\put{$a_s(\Delta_n)$} at 0 -1.9
	\put{$\smallmatrix  0\le s < n \endsmallmatrix$} at 0 -2.25
	\put{$\frac{n-s+1}{n+1}\binom{n+s}s$} at 1 -2
	\put{$\binom{n+s-1}s$} at 2 -2
	\put{$\left[{\smallmatrix n+s-2\cr s\endsmallmatrix}\right]$} at 3 -2
	\put{\rm see Section \ref{exceptional}} at 5.5 -2
	\setdots <1mm>
	\plot 3.8 -2.05 4.8 -2.05 /
	\plot 6.2 -2.05 7.1 -2.05 /

	\put{$a(\Delta_n)$} at 0 -3
	\put{$\frac1{n+2}\binom{2n+2}{n+1}$} at 1 -3
	\put{$\binom{2n}n$} at 2 -3
	\put{$\left[{\smallmatrix 2n-1\cr n-1\endsmallmatrix}\right]$} at 3 -3
	\put{$833$} at 4 -3
	\put{$4\,160$} at 4.8 -3
	\put{$25\,080$} at 5.6 -3
	\put{$105$} at 6.4 -3
	\put{$8$} at 7 -3

	\endpicture}
	$$
\end{theorem}

\begin{remark} By analogy with the Bailey notation
$\left[\smallmatrix t\cr s\endsmallmatrix\right]$ one may be tempted to
introduce the following notation for the Catalan triangle:
$\left]\smallmatrix t\cr s\endsmallmatrix\right[ = \frac{t-2s+1}{t-s+1}\binom t s$.
Then the numbers for the case $\mathbb A$ are written as follows:
$$a_n(\mathbb A_n) =  \left]\smallmatrix 2n\cr n\endsmallmatrix\right[,\
a_s(\mathbb A_n) = \left]\smallmatrix n+s\cr s\endsmallmatrix\right[,\
a(\mathbb A) = \left]\smallmatrix 2n+2\cr n+1\endsmallmatrix\right[.$$
\end{remark}

\begin{remark}  The reader should observe that for 
 $\mathbb A_n$ and  
$\mathbb B_n$, the formula given for $a_s(\Delta_n)$ and $0\le s < n$
 works also for $s = n$. This is not the case for $\mathbb D_n$:
 whereas $\binom{2n-2}{n-2} = \binom{2n-2}n$, the numbers
  $\left[{\smallmatrix 2n-2\cr n-2\endsmallmatrix}\right]$ and $\left[{\smallmatrix 2n-2\cr n\endsmallmatrix}\right]$
  are different (the difference will be highlighted at the end of Section \ref{triangles}).
The Lucas triangle consists of the
numbers $\left[{\smallmatrix t\cr s\endsmallmatrix}\right]$ for all $0 \le s \le t$;
it therefore uses the
numbers $\left[{\smallmatrix 2n-2\cr n\endsmallmatrix}\right]$ at the positions, whereas the $\mathbb D$-triangle
(which we will now consider)
uses the numbers $\left[{\smallmatrix 2n-2\cr n-2\endsmallmatrix}\right]$.
\end{remark}

\subsection
{The triangles $\mathbb A, \mathbb B, \mathbb D$} 
The non-zero numbers $a_s(\Delta_n)$ for $\Delta = \mathbb A, \mathbb B, \mathbb D$ 
yield three triangles having similar properties. We will exhibit them 
in Section \ref{triangles}; see the triangles 
\ref{triangle1}, \ref{triangle2}, \ref{triangle3}.
The triangle \ref{triangle1} of type $\mathbb A$ is the Catalan triangle itself;
this is \seqnum{A009766} in Sloane's OEIS \cite{[S]}. The triangle \ref{triangle2} of type 
$\mathbb B$ is the triangle \seqnum{A059481}, 
corresponding to the increasing part of the Pascal triangle
(thus it consists of the binomial coefficients $\binom ts$ with $2s\le t+1$).
The triangle \ref{triangle3} of type $\mathbb D$ is
an expansion of the increasing part of
the Lucas triangle  \seqnum{A029635}.
Taking the increasing part of the rows
in the Lucas triangle (thus the numbers $\left[\smallmatrix t\cr s \endsmallmatrix\right]$ with
$2s\le t+1$), we obtain numbers which occur in the triangle of type $\mathbb D$,
namely the numbers $a_s(\mathbb D_n)$
with $0 \le s < n$; the numbers $a_n(\mathbb D_n)$ on the diagonal however 
are given by a similar, but deviating formula
(they are listed as the sequence \seqnum{A129869}).
The Lucas triangle is \seqnum{A029635}, but the triangle $\mathbb D$ itself
was, at the time of the writing, not yet recorded in OEIS; now it is \seqnum{A241188}.
 
We see that the entries $a_s(n)$ of the triangles $\mathbb A$ and $\mathbb B$, as well as those of
the lower triangular part of the triangle $\mathbb D$ can be obtained in a unified way
from three triangles with entries $z_s(t)$ which satisfy the following recursion formula
$$
  z_s(t) = z_{s-1}(t-1)+z_s(t-1)
  $$
(they are exhibited in Section \ref{triangles} as triangles S~\ref{triangle1-S},
S~\ref{triangle2-S}, S~\ref{triangle3-S} using the shearing $a_s(n) = z_s(n+s-1)$). 
The recursion formula can be rewritten as $z_s(t) = \sum_{i= 0}^{s}z_i(t-s+i+1)$
(sometimes called the hockey stick formula). A consequence of the hockey stick formula is the fact
that summing up the rows of any of the three triangles $\mathbb A, \mathbb B, \mathbb D$,
we again obtain numbers
which appear in the triangle. 

Let us provide further details on the triangles to be sheared.
Consider first the case $\mathbb B$. Here we start with the Pascal triangle; 
thus we deal with the triangle with numbers $z_s(t) = \binom ts$ and the
initial conditions are $z_0(t) = z_t(t) = 1$ for all $t\ge 0$.
In case $\mathbb D$, we start with the Lucas triangle with numbers
$z_s(t) = \left[\smallmatrix t\cr s \endsmallmatrix\right]$, and
the initial conditions are
$z_0(t) = 1,\ z_t(t) = 2$
for all $t\ge 1$ 
(these initial conditions are the reason for calling the Lucas triangle also the
$(1,2)$-triangle). 
In the case $\mathbb A$ we start with a sheared Catalan triangle, and here 
the initial conditions are $z_0(t) = 1$
and $z_{t+1}(2t) = 0$ for  all $t\ge 0$.
		
\subsection{Related results}
Let us repeat that in this paper
$a_n(\Delta_n)$ denotes the number of tilting modules, $a(\Delta_n)$ the
number of support-tilting modules, for $\Lambda$ of Dynkin type $\Delta_n$.
As we have mentioned, the relevance of the numbers 
$a_n(\Delta_n)$ and $a(\Delta_n)$ was not fully realized in the eighties. 
It became apparent through the work of Fomin and
Zelevinksy when dealing with cluster algebras and the corresponding cluster complexes 
(see in particular \cite{[FZ]} and \cite{[FR]}):
the numbers $a_n(\Delta_n)$ and $a(\Delta_n)$ appear in \cite{[FZ]} 
as the numbers $N(\Delta_n)$ of clusters and $N^+(\Delta_n)$ of positive clusters, respectively
(see Propositions 3.8 and 3.9 of  \cite{[FZ]}). For the numbers $a_s(\Delta_n)$ in general, see
Chapoton \cite{[C]} in case $\mathbb A$ and $\mathbb B$ and Krattentaler \cite{[Kt]} in case
$\mathbb D$. A conceptual proof of the equalities $a(\Delta_n) = N(\Delta_n)$ 
and $a_n(\Delta_n) = N^+(\Delta_n)$ has been given
by Ingalls and Thomas \cite{[IT]} in case $\Delta_n$ is simply laced (thus of type $\mathbb A, \mathbb D$ of $\mathbb E$). 
The considerations of Ingalls and Thomas have been extended by the authors \cite{[ONFR]}
to the non-simply laced cases.
The papers \cite{[IT]} and \cite{[ONFR]}
show in which way the representation theory of hereditary artin algebras can be used in order
to categorify the cluster complex of Fomin and Zelevinsky: this is the reason for the equalities. 
Another method to relate clusters and support tilting modules is due to Marsh, Reineke and
Zelevinsky \cite{[MRZ]}. Finally, let us stress that also the Coxeter diagrams 
$\mathbb H_3$ and $\mathbb H_4$ can be treated in a similar way, using
hereditary artinian rings which are not artin algebras; this will be shown in \cite{[FR]}.

The main result
of the present paper is the direct calculation of the numbers 
$a_n(\Delta_n)$ in the case $\Delta = \mathbb B$;
see Section \ref{tilting}. Of course, using
\cite{[ONFR]}, this calculation can be replaced by referring to the 
determination of the corresponding cluster numbers by Fomin and Zelevinsky in \cite{[FZ]}.
On the other hand, we hope that our proof is of interest in itself. 

There is an independent development which has to be mentioned, 
namely the theory of generalized non-crossing
partitions (see for example \cite{[A]}). 
It is the Ingalls-Thomas paper \cite{[IT]} (and \cite{[ONFR]}; see also
the survey \cite{[R2]}) which provides the basic setting for using the representation theory
of a hereditary artin algebra $\Lambda$ in order to deal with non-crossing partitions.
It turns out that there is a large number of counting problems for $\mo\Lambda$
which yield the same answer, namely the numbers $a(\Delta_n)$ and $a_s(\Delta_n)$. 
For example, $a_s(\Delta)$ is also the number of antichains in $\mo\Lambda$ of size $s$:
an {\it antichain} $A = \{A_1,\dots,A_t\}$ in $\mo\Lambda$ is a set of pairwise
orthogonal bricks (a brick is a module whose endomorphism ring is a division ring, 
and two bricks $A_1, A_2$ are said to be orthogonal provided $\Hom(A_1,A_2) = 0 = \Hom(A_2,A_1)$; antichains are called discrete subsets in \cite{[GP]} and $\Hom$-free subsets in \cite{[HK]}). 

Since the support-tilting modules for a Dynkin algebra of Dynkin type $\Delta$
correspond bijectively to the non-crossing partitions
of type $\Delta$, the calculations presented here may be considered as a
categorification of results concerning non-crossing partitions (for a general outline
see Hubery-Krause \cite{[HK]}). Finally, let us mention that there is a corresponding 
discussion of the number of
ad-nilpotent ideals of a Borel subalgebra of a simple Lie algebra;
see Panyushev \cite{[P]}.
	
\subsection
{\bf Outline of the paper} 
Let us stress again that there is an inductive procedure using
the hook formula (Proposition \ref{hook-formula}) and a modified hook formula
(Proposition \ref{modified})
in order to obtain the numbers $a_s(\mathbb A_n)$ for $0\le s \le n$,
as well as the numbers
$a_s(\Delta_n)$ for $\Delta = \mathbb B, \mathbb D$ for $0 \le s < n$, 
provided we know the numbers
$a_n(\Delta_n)$. As we have mentioned,
for the numbers $a_n(\mathbb D_n)$ we may refer to \cite{[BLR]}. 
In Section \ref{algebras}, we will show that
the numbers $a_n(\mathbb B_n)$ and $a_n(\mathbb C_n)$ coincide;
thus it remains to determine the numbers $a_n(\mathbb B_n)$.
This will be done in Section \ref{tilting}. In Section \ref{summation}, we calculate $a(\Delta_n)$ for
$\Delta = \mathbb A, \mathbb B, \mathbb D$. 

Section \ref{triangles} presents the triangles $\mathbb A, \mathbb B, \mathbb D$
as well as the corresponding Catalan, Pascal, and Lucas triangles, and some observations
concerning repetition of numbers in the triangles are recorded. Section \ref{exceptional}
provides the numbers $a_s(\Delta_n)$ for the exceptional cases 
$\Delta_n = \mathbb E_6,  \mathbb E_7,  \mathbb E_8,  \mathbb F_4,  \mathbb G_2$.

\section{The triangles}\label{triangles}

\begin{tri}\label{triangle1}{\bf The  triangle of type $\mathbb A$; this is \seqnum{A009766}}
$$
\hbox{\beginpicture
	\setcoordinatesystem units <1cm,.45cm>
	\put{$a_s(\mathbb A_n) = \dfrac{n-s+1}{n+1}\dbinom{n+s}s$} at 7 -1.5
	\multiput{1} at 0 0  0 -1  0 -2  0 -3  0 -4  0 -5  0 -6  0 -7  0 -8  0 -9 /
	\put{1} at 1 -1
	\put{2} at 1 -2
	\put{2} at 2 -2
	\put{3} at 1 -3
	\put{5} at 2 -3
	\put{5} at 3 -3
	\put{4} at 1 -4
	\put{9} at 2 -4
	\put{14} at 3 -4
	\put{14} at 4 -4
	\put{5} at 1 -5
	\put{14} at 2 -5
	\put{28} at 3 -5
	\put{42} at 4 -5
	\put{42} at 5 -5
	\put{6} at 1 -6
	\put{20} at 2 -6
	\put{48} at 3 -6
	\put{90} at 4 -6
	\put{132} at 5 -6
	\put{132} at 6 -6
	\put{7} at 1 -7
	\put{27} at 2 -7
	\put{75} at 3 -7
	\put{165} at 4 -7
	\put{297} at 5 -7
	\put{429} at 6 -7
	\put{429} at 7 -7
	\put{8} at 1 -8
	\put{35} at 2 -8
	\put{110} at 3 -8
	\put{275} at 4 -8
	\put{572} at 5 -8
	\put{1001} at 6 -8
	\put{1430} at 7 -8
	\put{1430} at 8 -8
	\put{9} at 1 -9
	\put{44} at 2 -9
	\put{154} at 3 -9
	\put{429} at 4 -9
	\put{1001} at 5 -9
	\put{2002} at 6 -9
	\put{3432} at 7 -9
	\put{4862} at 8 -9
	\put{4862} at 9 -9

	\put{$\ssize n$} at -2 1
	\put{$\ssize0$} at -2 0
	\put{$\ssize1$} at -2 -1
	\put{$\ssize2$} at -2 -2
	\put{$\ssize3$} at -2 -3
	\put{$\ssize4$} at -2 -4
	\put{$\ssize5$} at -2 -5
	\put{$\ssize6$} at -2 -6
	\put{$\ssize7$} at -2 -7
	\put{$\ssize8$} at -2 -8
	\put{$\ssize9$} at -2 -9

	\plot -2 1.5  -1.5 1 /
	\put{$\ssize s$} at -1.5 1.5
	\put{$\ssize0$} at 0 1.5
	\put{$\ssize1$} at 1 1.5
	\put{$\ssize2$} at 2 1.5
	\put{$\ssize3$} at 3 1.5
	\put{$\ssize4$} at 4 1.5
	\put{$\ssize5$} at 5 1.5
	\put{$\ssize6$} at 6 1.6
	\put{$\ssize7$} at 7 1.5
	\put{$\ssize8$} at 8 1.5
	\put{$\ssize9$} at 9 1.5

	\put{} at 0 -10
	\put{sum} [r]  at 11.5 1.5
	\put{1} [r] at 11.5 0
	\put{2}  [r] at 11.5 -1
	\put{5}  [r] at 11.5 -2
	\put{14}  [r] at 11.5 -3
	\put{42}  [r] at 11.5 -4
	\put{132}  [r] at 11.5 -5
	\put{429}  [r] at 11.5 -6
	\put{1430}   [r] at 11.5 -7
	\put{4862}  [r] at 11.5 -8
	\put{16796}  [r] at 11.5 -9
	\endpicture}
$$
\end{tri}

\begin{tri}\label{triangle2}{\bf The  triangle of type $\mathbb B$; this is \seqnum{A059481}}
$$
\hbox{\beginpicture
	\setcoordinatesystem units <1cm,.45cm>
	\put{$a_s(\mathbb B_n) = \dbinom{n+s-1}s$} at 7 -1.5
	\multiput{1} at   0 -1  0 -2  0 -3  0 -4  0 -5  0 -6  0 -7  0 -8  0 -9 /
	\put{1} at 0 0
	\put{1} at 11.4 0

	\put{1} at 1 -1
	\put{2} at 1 -2
	\put{3} at 2 -2
	\put{3} at 1 -3
	\put{6} at 2 -3
	\put{10} at 3 -3
	\put{4} at 1 -4
	\put{10} at 2 -4
	\put{20} at 3 -4
	\put{35} at 4 -4
	\put{5} at 1 -5
	\put{15} at 2 -5
	\put{35} at 3 -5
	\put{70} at 4 -5
	\put{126} at 5 -5
	\put{6} at 1 -6
	\put{21} at 2 -6
	\put{56} at 3 -6
	\put{126} at 4 -6
	\put{252} at 5 -6
	\put{462} at 6 -6
	\put{7} at 1 -7
	\put{28} at 2 -7
	\put{84} at 3 -7
	\put{210} at 4 -7
	\put{462} at 5 -7
	\put{924} at 6 -7
	\put{1716} at 7 -7
	\put{8} at 1 -8
	\put{36} at 2 -8
	\put{120} at 3 -8
	\put{330} at 4 -8
	\put{792} at 5 -8
	\put{1716} at 6 -8
	\put{3432} at 7 -8
	\put{6435} at 8 -8
	\put{9} at 1 -9
	\put{45} at 2 -9
	\put{165} at 3 -9
	\put{495} at 4 -9
	\put{1287} at 5 -9
	\put{3003} at 6 -9
	\put{6435} at 7 -9
	\put{12870} at 8 -9
	\put{24310} at 9 -9

	\put{$\ssize n$} at -2 1
	\put{$\ssize0$} at -2 0
	\put{$\ssize1$} at -2 -1
	\put{$\ssize2$} at -2 -2
	\put{$\ssize3$} at -2 -3
	\put{$\ssize4$} at -2 -4
	\put{$\ssize5$} at -2 -5
	\put{$\ssize6$} at -2 -6
	\put{$\ssize7$} at -2 -7
	\put{$\ssize8$} at -2 -8
	\put{$\ssize9$} at -2 -9

	\plot -2 1.5  -1.5 1 /
	\put{$\ssize s$} at -1.5 1.5
	\put{$\ssize0$} at 0 1.5
	\put{$\ssize1$} at 1 1.5
	\put{$\ssize2$} at 2 1.5
	\put{$\ssize3$} at 3 1.5
	\put{$\ssize4$} at 4 1.5
	\put{$\ssize5$} at 5 1.5
	\put{$\ssize6$} at 6 1.6
	\put{$\ssize7$} at 7 1.5
	\put{$\ssize8$} at 8 1.5
	\put{$\ssize9$} at 9 1.5

	\put{} at 0 -10
	\put{sum} [r]  at 11.5 1.5
	\put{} [r] at 11.5 0
	\put{2}  [r] at 11.5 -1
	\put{6}  [r] at 11.5 -2
	\put{20}  [r] at 11.5 -3
	\put{70}  [r] at 11.5 -4
	\put{252}  [r] at 11.5 -5
	\put{924}  [r] at 11.5 -6
	\put{3432}   [r] at 11.5 -7
	\put{12870}  [r] at 11.5 -8
	\put{48620}  [r] at 11.5 -9
	\endpicture}
	$$
\end{tri}

\begin{tri}\label{triangle3}{\bf The  triangle  of type $\mathbb D$; this is now 
\seqnum{A241188}}
$$
\hbox{\beginpicture
	\setcoordinatesystem units <1cm,.45cm>
	\put{$a_s(\mathbb D_n) \,= \left\{\begin{matrix} \cr\cr\cr\end{matrix}\right.$} [l] at 4.2 -1.4
	
\put{$\left[ \smallmatrix{n+s-2}\cr s \endsmallmatrix\right]$ \ for $0\le s <n;
      $} [l] at 6.4 -.3
      \put{$ \left[ \smallmatrix{2n-2}\cr n-2 \endsmallmatrix\right]$\qquad for $s=n$.} [l] at 6.4 -2.5

\multiput{1} at  0 -2  0 -3  0 -4  0 -5  0 -6  0 -7  0 -8  0 -9 /
\multiput{$\cdot$} at 0 0  0 -1  1 -1  /
\multiput{$\cdot$} at 11.4 0  11.4 -1  /
\put{} at 1 -1
\put{2} at 1 -2
\put{1} at 2 -2
\put{3} at 1 -3
\put{5} at 2 -3
\put{5} at 3 -3
\put{4} at 1 -4
\put{9} at 2 -4
\put{16} at 3 -4
\put{20} at 4 -4
\put{5} at 1 -5
\put{14} at 2 -5
\put{30} at 3 -5
\put{55} at 4 -5
\put{77} at 5 -5
\put{6} at 1 -6
\put{20} at 2 -6
\put{50} at 3 -6
\put{105} at 4 -6
\put{196} at 5 -6
\put{294} at 6 -6
\put{7} at 1 -7
\put{27} at 2 -7
\put{77} at 3 -7
\put{182} at 4 -7
\put{378} at 5 -7
\put{714} at 6 -7
\put{1122} at 7 -7
\put{8} at 1 -8
\put{35} at 2 -8
\put{112} at 3 -8
\put{294} at 4 -8
\put{672} at 5 -8
\put{1386} at 6 -8
\put{2640} at 7 -8
\put{4290} at 8 -8
\put{9} at 1 -9
\put{44} at 2 -9
\put{156} at 3 -9
\put{450} at 4 -9
\put{1122} at 5 -9
\put{2508} at 6 -9
\put{5148} at 7 -9
\put{9867} at 8 -9
\put{16445} at 9 -9

\put{$\ssize n$} at -2 1
\put{$\ssize0$} at -2 0
\put{$\ssize1$} at -2 -1
\put{$\ssize2$} at -2 -2
\put{$\ssize3$} at -2 -3
\put{$\ssize4$} at -2 -4
\put{$\ssize5$} at -2 -5
\put{$\ssize6$} at -2 -6
\put{$\ssize7$} at -2 -7
\put{$\ssize8$} at -2 -8
\put{$\ssize9$} at -2 -9

\plot -2 1.5  -1.5 1 /
\put{$\ssize s$} at -1.5 1.5
\put{$\ssize0$} at 0 1.5
\put{$\ssize1$} at 1 1.5
\put{$\ssize2$} at 2 1.5
\put{$\ssize3$} at 3 1.5
\put{$\ssize4$} at 4 1.5
\put{$\ssize5$} at 5 1.5
\put{$\ssize6$} at 6 1.6
\put{$\ssize7$} at 7 1.5
\put{$\ssize8$} at 8 1.5
\put{$\ssize9$} at 9 1.5

\put{} at 0 -10
\put{sum} [r]  at 11.5 1.5
\put{} [r] at 11.5 0
\put{}  [r] at 11.5 -1
\put{4}  [r] at 11.5 -2
\put{14}  [r] at 11.5 -3
\put{50}  [r] at 11.5 -4
\put{182}  [r] at 11.5 -5
\put{672}  [r] at 11.5 -6
\put{2508}   [r] at 11.5 -7
\put{9438}  [r] at 11.5 -8
\put{35750}  [r] at 11.5 -9

\setdashes <1mm>
\plot 1.5 -1.5  1.5 -2.5
      2.5 -2.5  2.5 -3.5
            3.5 -3.5  3.5 -4.5
	          4.5 -4.5  4.5 -5.5
		        5.5 -5.5  5.5 -6.5
			      6.5 -6.5  6.5 -7.5
			            7.5 -7.5  7.5 -8.5
				          8.5 -8.5  8.5 -9.5
					        9.5 -9.5  9.5 -10
						/
	\endpicture}
$$
\end{tri}

\begin{triS}\label{triangle1-S}{\bf The sheared Catalan triangle \seqnum{A008315}}
$$
\hbox{\beginpicture
	\setcoordinatesystem units <.9cm,.45cm>
	\put{$\dbinom t s- \dbinom t {s-1} = \dfrac{t-2s+1}{t-s-1}\dbinom ts$} at 7  -2
	\multiput{1} at  0 0  0 -1  0 -2  0 -3  0 -4  0 -5  0 -6  0 -7  0 -8  0 -9 /
	\put{1} at 1 -1
	\put{2} at 1 -2
	\put{} at 2 -2
	\put{3} at 1 -3
	\put{2} at 2 -3
	\put{} at 3 -3
	\put{4} at 1 -4
	\put{5} at 2 -4
	\put{} at 3 -4
	\put{} at 4 -4
	\put{5} at 1 -5
	\put{9} at 2 -5
	\put{5} at 3 -5
	\put{} at 4 -5
	\put{} at 5 -5
	\put{6} at 1 -6
	\put{14} at 2 -6
	\put{14} at 3 -6
	\put{} at 4 -6
	\put{} at 5 -6
	\put{} at 6 -6
	\put{7} at 1 -7
	\put{20} at 2 -7
	\put{28} at 3 -7
	\put{14} at 4 -7
	\put{} at 5 -7
	\put{} at 6 -7
	\put{} at 7 -7
	\put{8} at 1 -8
	\put{27} at 2 -8
	\put{48} at 3 -8
	\put{42} at 4 -8
	\put{} at 5 -8
	\put{} at 6 -8
	\put{} at 7 -8
	\put{} at 8 -8
	\put{9} at 1 -9
	\put{35} at 2 -9
	\put{75} at 3 -9
	\put{90} at 4 -9
	\put{42} at 5 -9
	\put{} at 6 -9
	\put{} at 7 -9
	\put{} at 8 -9
	\put{} at 9 -9

	\put{$\ssize t$} at -2 1
	\put{$\ssize0$} at -2 0
	\put{$\ssize1$} at -2 -1
	\put{$\ssize2$} at -2 -2
	\put{$\ssize3$} at -2 -3
	\put{$\ssize4$} at -2 -4
	\put{$\ssize5$} at -2 -5
	\put{$\ssize6$} at -2 -6
	\put{$\ssize7$} at -2 -7
	\put{$\ssize8$} at -2 -8
	\put{$\ssize9$} at -2 -9

	\plot -2 1.5  -1.5 1 /
	\put{$\ssize s$} at -1.5 1.5
	\put{$\ssize0$} at 0 1.5
	\put{$\ssize1$} at 1 1.5
	\put{$\ssize2$} at 2 1.5
	\put{$\ssize3$} at 3 1.5
	\put{$\ssize4$} at 4 1.5
	\put{$\ssize5$} at 5 1.5
	\put{$\ssize6$} at 6 1.6
	\put{$\ssize7$} at 7 1.5
	\put{$\ssize8$} at 8 1.5
	\put{$\ssize9$} at 9 1.5
	\plot 5.5 -10  5.5 -8.5  4.5 -8.5  4.5 -6.5  3.5 -6.5 3.5 -4.5  2.5 -4.5 2.5 -2.5  1.5 -2.5
	   1.5 -0.5  0.5 -0.5  0.5 .5 /
	   \setdots <.5mm>
	   \plot  0 0   1 -1 /
	   \plot  0 -1  2 -3 /
	   \plot  0 -2  3 -5 /
	   \plot  0 -3  4 -7 /
	   \plot  0 -4  5 -9 /
	   \plot  0 -5  5 -10 /
	   \plot  0 -6  4 -10 /
	   \plot  0 -7  3 -10 /
	   \plot  0 -8  2 -10 /
	   \plot  0 -9  1 -10 /

	   \put{} at 0 -10
	   \endpicture}
	   $$
\end{triS}
	   
\begin{triS}\label{triangle2-S}{\bf The  
Pascal triangle \seqnum{A007318}, left of the staircase line is the increasing part}
$$
   \hbox{\beginpicture
		\setcoordinatesystem units <.9cm,.45cm>
		\put{$\dbinom t s$} at 7  -2
		\multiput{1} at 0 0  0 -1  0 -2  0 -3  0 -4  0 -5  0 -6  0 -7  0 -8  0 -9 /
		\put{1} at 1 -1
		\put{2} at 1 -2
		\put{1} at 2 -2
		\put{3} at 1 -3
		\put{3} at 2 -3
		\put{1} at 3 -3
		\put{4} at 1 -4
		\put{6} at 2 -4
		\put{4} at 3 -4
		\put{1} at 4 -4
		\put{5} at 1 -5
		\put{10} at 2 -5
		\put{10} at 3 -5
		\put{5} at 4 -5
		\put{1} at 5 -5
		\put{6} at 1 -6
		\put{15} at 2 -6
		\put{20} at 3 -6
		\put{15} at 4 -6
		\put{6} at 5 -6
		\put{1} at 6 -6
		\put{7} at 1 -7
		\put{21} at 2 -7
		\put{35} at 3 -7
		\put{35} at 4 -7
		\put{21} at 5 -7
		\put{7} at 6 -7
		\put{1} at 7 -7
		\put{8} at 1 -8
		\put{28} at 2 -8
		\put{56} at 3 -8
		\put{70} at 4 -8
		\put{56} at 5 -8
		\put{28} at 6 -8
		\put{8} at 7 -8
		\put{1} at 8 -8
		\put{9} at 1 -9
		\put{36} at 2 -9
		\put{84} at 3 -9
		\put{126} at 4 -9
		\put{126} at 5 -9
		\put{84} at 6 -9
		\put{36} at 7 -9
		\put{9} at 8 -9
		\put{1} at 9 -9

		\put{$\ssize t$} at -2 1
		\put{$\ssize0$} at -2 0
		\put{$\ssize1$} at -2 -1
		\put{$\ssize2$} at -2 -2
		\put{$\ssize3$} at -2 -3
		\put{$\ssize4$} at -2 -4
		\put{$\ssize5$} at -2 -5
		\put{$\ssize6$} at -2 -6
		\put{$\ssize7$} at -2 -7
		\put{$\ssize8$} at -2 -8
		\put{$\ssize9$} at -2 -9

		\plot -2 1.5  -1.5 1 /
		\put{$\ssize s$} at -1.5 1.5
		\put{$\ssize0$} at 0 1.5
		\put{$\ssize1$} at 1 1.5
		\put{$\ssize2$} at 2 1.5
		\put{$\ssize3$} at 3 1.5
		\put{$\ssize4$} at 4 1.5
		\put{$\ssize5$} at 5 1.5
		\put{$\ssize6$} at 6 1.6
		\put{$\ssize7$} at 7 1.5
		\put{$\ssize8$} at 8 1.5
		\put{$\ssize9$} at 9 1.5
		\plot 5.5 -10  5.5 -8.5  4.5 -8.5  4.5 -6.5  3.5 -6.5 3.5 -4.5  2.5 -4.5 2.5 -2.5  1.5 -2.5
		   1.5 -0.5  0.5 -0.5  0.5 0.5 /
		   \setdots <.5mm>
		   \plot  0 0   1 -1 /
		   \plot  0 -1  2 -3 /
		   \plot  0 -2  3 -5 /
		   \plot  0 -3  4 -7 /
		   \plot  0 -4  5 -9 /
		   \plot  0 -5  5 -10 /
		   \plot  0 -6  4 -10 /
		   \plot  0 -7  3 -10 /
		   \plot  0 -8  2 -10 /
		   \plot  0 -9  1 -10 /

		   \put{} at 0 -10
		   \endpicture}
$$
\end{triS}

\begin{triS}\label{triangle3-S}{\bf The Lucas 
triangle \seqnum{A029635}, left of the staircase line is the increasing part}
$$
		   \hbox{\beginpicture
			\setcoordinatesystem units <.9cm,.45cm>
			\put{$\left[\begin{matrix} t\cr s \end{matrix}\right]$} at 7  -2
			\multiput{1} at  0 -1  0 -2  0 -3  0 -4  0 -5  0 -6  0 -7  0 -8  0 -9 /
			\put{$\cdot$} at 0 0
			\put{2} at 1 -1
			\put{3} at 1 -2
			\put{2} at 2 -2
			\put{4} at 1 -3
			\put{5} at 2 -3
			\put{2} at 3 -3
			\put{5} at 1 -4
			\put{9} at 2 -4
			\put{7} at 3 -4
			\put{2} at 4 -4
			\put{6} at 1 -5
			\put{14} at 2 -5
			\put{16} at 3 -5
			\put{9} at 4 -5
			\put{2} at 5 -5
			\put{7} at 1 -6
			\put{20} at 2 -6
			\put{30} at 3 -6
			\put{25} at 4 -6
			\put{11} at 5 -6
			\put{2} at 6 -6
			\put{8} at 1 -7
			\put{27} at 2 -7
			\put{50} at 3 -7
			\put{55} at 4 -7
			\put{36} at 5 -7
			\put{13} at 6 -7
			\put{2} at 7 -7
			\put{9} at 1 -8
			\put{35} at 2 -8
			\put{77} at 3 -8
			\put{105} at 4 -8
			\put{91} at 5 -8
			\put{49} at 6 -8
			\put{15} at 7 -8
			\put{2} at 8 -8
			\put{10} at 1 -9
			\put{44} at 2 -9
			\put{112} at 3 -9
			\put{182} at 4 -9
			\put{196} at 5 -9
			\put{140} at 6 -9
			\put{64} at 7 -9
			\put{17} at 8 -9
			\put{2} at 9 -9

			\put{$\ssize t$} at -2 1
			\put{$\ssize0$} at -2 0
			\put{$\ssize1$} at -2 -1
			\put{$\ssize2$} at -2 -2
			\put{$\ssize3$} at -2 -3
			\put{$\ssize4$} at -2 -4
			\put{$\ssize5$} at -2 -5
			\put{$\ssize6$} at -2 -6
			\put{$\ssize7$} at -2 -7
			\put{$\ssize8$} at -2 -8
			\put{$\ssize9$} at -2 -9

			\plot -2 1.5  -1.5 1 /
			\put{$\ssize s$} at -1.5 1.5
			\put{$\ssize0$} at 0 1.5
			\put{$\ssize1$} at 1 1.5
			\put{$\ssize2$} at 2 1.5
			\put{$\ssize3$} at 3 1.5
			\put{$\ssize4$} at 4 1.5
			\put{$\ssize5$} at 5 1.5
			\put{$\ssize6$} at 6 1.6
			\put{$\ssize7$} at 7 1.5
			\put{$\ssize8$} at 8 1.5
			\put{$\ssize9$} at 9 1.5
			\plot 5.5 -10  5.5 -8.5  4.5 -8.5  4.5 -6.5  3.5 -6.5 3.5 -4.5  2.5 -4.5 2.5 -2.5  1.5 -2.5
			   1.5 -0.5  0.5 -0.5 0.5 0 /
			   \setdots <.5mm>
			   \plot  0 -1  2 -3 /
			   \plot  0 -2  3 -5 /
			   \plot  0 -3  4 -7 /
			   \plot  0 -4  5 -9 /
			   \plot  0 -5  5 -10 /
			   \plot  0 -6  4 -10 /
			   \plot  0 -7  3 -10 /
			   \plot  0 -8  2 -10 /
			   \plot  0 -9  1 -10 /

			   \put{} at 0 -10
			   \endpicture}
$$
\end{triS}			   

\begin{remark} In the triangle \ref{triangle3} of type $\mathbb D$ and 
in the corresponding Lucas triangle S \ref{triangle3-S}
some values are left open (this is indicated by a dot).

In the Lucas triangle S \ref{triangle3-S}, this concerns the value at the position $(0,0)$ which
could be denoted as $\left[\smallmatrix 0\cr 0 \endsmallmatrix\right]$.
The value should be one of the numbers
$1$ or $2$ (in OEIS \seqnum{A029635},  the number is chosen to be $2$). Note that here we deal with
the product $\frac00\binom00$: whereas $\binom 0 0 = 1$ is well-defined, there is  
the ambiguous fraction $\frac 0 0$.

In the triangle \ref{triangle3} of type $\mathbb D$, the positions $(0,0), (0,1), (1,1)$ are left
open, since the series of Dynkin diagrams $\mathbb D_n$ starts with $n = 2$ (but see \seqnum{A129869});
by definition $\mathbb D_2 = \mathbb A_1\sqcup \mathbb A_1$ and $\mathbb D_3 = \mathbb A_3$.
As a consequence, also the corresponding entries in the summation sequence are missing.
\end{remark}
  
\setcounter{subsection}{3}
\subsection
{\bf Some observations concerning the triangles $\mathbb A, \mathbb B, \mathbb D$}
     
\begin{proposition} The sum sequence occurs as a diagonal.
\end{proposition}

      In the $\mathbb A$-triangle, the sum sequence is the same sequence as the main diagonal 
(and these are just the Catalan numbers):
$$
       a(\mathbb A_n) = a_{n+1}(\mathbb A_{n+1}).
$$

       \noindent
       In the $\mathbb B$-triangle, the sum sequence is the same sequence as the second  diagonal
$$
        a(\mathbb B_n) = a_{n}(\mathbb B_{n+1}). \ \ 
$$

	\noindent
	In the $\mathbb D$-triangle, the sum sequence is the same sequence as the fourth diagonal
$$
	 a(\mathbb D_n) = a_{n-1}(\mathbb D_{n+2}).
$$
		
\begin{proposition} The main diagonal uses the same sequence as one of the other diagonals.
\end{proposition}

\noindent
In the $\mathbb A$-triangle, this concerns the main diagonal and the second diagonal:
$$
 a_n(\mathbb A_n) = a_{n-1}(\mathbb A_{n}). \ \
$$
 In the $\mathbb B$-triangle, this concerns the main diagonal and the second diagonal:
$$
  a_n(\mathbb B_n) = a_{n-1}(\mathbb B_{n+1}).
$$
  In the $\mathbb D$-triangle, this concerns the main diagonal and the fifth diagonal:
$$
   a_n(\mathbb D_n) = a_{n-2}(\mathbb D_{n+2}).
$$

It may be of interest to exhibit explicit bijections between the corresponding sets
of support-tilting modules. It seems that only in the case $\mathbb A$, this can be done easily
(see Remark \ref{type-a}).

\subsection
{\bf Comparison between the Lucas triangle and the $\mathbb D$-triangle}

The difference between the number $\left[{\begin{matrix} 2n-2\cr n\end{matrix}}\right]$
and $\left[{\begin{matrix} 2n-2\cr n-2\end{matrix}}\right]$ seems to be of interest:
$$
      \left[{\begin{matrix} 2n-2\cr n\end{matrix}}\right]
        - \left[{\begin{matrix} 2n-2\cr n-2\end{matrix}}\right]
	   = \frac{1}{n}\binom{2n-2}{n-1}.
$$
This means the following:

\begin{proposition}\label{comparision}
	   $$
	      \left[\begin{matrix} 2n-2\cr n \end{matrix}\right] - a_n(\mathbb D_n) = a_{n-1}(\mathbb A_{n-1}).
	      $$
\end{proposition}

\begin{proof} We show that
$$
	        \frac{3n-4}{n}\binom{2n-2}{n-2}
		  +\frac{1}{n}\binom{2n-2}{n-1} \ =\
		    \frac{3n-2}{2n-2}\binom{2n-2}{n}
$$
We rewrite
\begin{eqnarray*}
		      \binom{2n-2}{n-2} &=& \frac{n}{2n-2} \binom{2n-2}{n}, \\
		        \binom{2n-2}{n-1} &=& \frac{n}{n-1}\binom{2n-2}{n}. 
\end{eqnarray*}
The assertion now follows from the equality
$$
 \frac{3n-4}{n}\cdot  \frac{n}{2n-2}\ +\ \frac{1}{n}\cdot  \frac{n}{n-1}\ =\ \frac{3n-2}{2n-2}.
$$
\end{proof}
	
\noindent
Here is a table of these numbers
$$
\hbox{\beginpicture
	\setcoordinatesystem units <1cm,.45cm>
	\put{$\ssize n$} at -2 1
	\put{$\left[{\begin{matrix} 2n-2\cr n\end{matrix}}\right]$}
	    [r] at 0.3 1
	    \put{$\left[{\begin{matrix} 2n-2\cr n-2\end{matrix}}\right]$}  [r] at 2.4 1
	    \put{$\dfrac{1}{n}\dbinom{2n-2}{n-1}$}  [r] at 4.9 1
	    \put{$\ssize2$} at -2 -2
	    \put{$\ssize3$} at -2 -3
	    \put{$\ssize4$} at -2 -4
	    \put{$\ssize5$} at -2 -5
	    \put{$\ssize6$} at -2 -6
	    \put{$\ssize7$} at -2 -7
	    \put{$\ssize8$} at -2 -8
	    \put{$\ssize9$} at -2 -9
	    \put{$2$} [r] at 0 -2
	    \put{$7$} [r]  at 0 -3
	    \put{$25$} [r]  at 0 -4
	    \put{$91$}  [r] at 0 -5
	    \put{$336$}  [r] at 0 -6
	    \put{$1254$}  [r] at 0 -7
	    \put{$4719$}  [r] at 0 -8
	    \put{$17875$}  [r] at 0 -9
	    \put{$1$}  [r] at 2 -2
	    \put{$5$}  [r] at 2 -3
	    \put{$20$}  [r] at 2 -4
	    \put{$77$}  [r] at 2 -5
	    \put{$294$}  [r] at 2 -6
	    \put{$1122$}  [r] at 2 -7
	    \put{$4290$}  [r] at 2 -8
	    \put{$16445$}  [r] at 2 -9
	    \put{$1$}  [r] at 4 -2
	    \put{$2$}  [r] at 4 -3
	    \put{$5$}  [r] at 4 -4
	    \put{$14$}  [r] at 4 -5
	    \put{$42$}  [r] at 4 -6
	    \put{$132$}  [r] at 4 -7
	    \put{$429$}  [r] at 4 -8
	    \put{$1430$}  [r] at 4 -9
	    \put{} at 0 -10
	    \endpicture}
$$

\begin{remark} Proposition \ref{comparision} is essentially the modified hook formula
for type $\mathbb D$
which will be presented in Proposition \ref{modified}:
the Lucas triangle uses the hook formula for the whole triangle, 
whereas the triangle $\mathbb D$ uses the modified hook formula on the subdiagonal.
\end{remark}
	    
\section{The exceptional cases}\label{exceptional}

Here are the numbers $a_s(\Delta_n)$ and $a(\Delta_n)$ in the exceptional cases
$\mathbb E_6, \mathbb E_7, \mathbb E_8, \mathbb F_4,$ and $\mathbb G_2$
(we add some suitable additional rows in order to stress the induction scheme):
$$
\hbox{\beginpicture
	\setcoordinatesystem units <.97cm,.45cm>
	\put{1} at 0 0
	\put{3} at 1 0
	\put{4} at 2 0
	\put{2} at 3 0

	\put{1} at 0 -1
	\put{4} at 1 -1
	\put{9} at 2 -1
	\put{14} at 3 -1
	\put{14} at 4 -1

	\put{1} at 0 -2
	\put{5} at 1 -2
	\put{14} at 2 -2
	\put{30} at 3 -2
	\put{55} at 4 -2
	\put{77} at 5 -2

	\put{1} at 0 -3
	\put{6} at 1 -3
	\put{20} at 2 -3
	\put{50} at 3 -3
	\put{110} at 4 -3
	\put{228} at 5 -3
	\put{418} at 6 -3

	\put{1} at 0 -4
	\put{7} at 1 -4
	\put{27} at 2 -4
	\put{77} at 3 -4
	\put{187} at 4 -4
	\put{429} at 5 -4
	\put{1001} at 6 -4
	\put{2431} at 7 -4

	\put{1} at 0 -5
	\put{8} at 1 -5
	\put{35} at 2 -5
	\put{112} at 3 -5
	\put{299} at 4 -5
	\put{728} at 5 -5
	\put{1771} at 6 -5
	\put{4784} at 7 -5
	\put{17342} at 8.1 -5

	\put{1} at 0 -7
	\put{3} at 1 -7
	\put{6} at 2 -7
	\put{10} at 3 -7

	\put{1} at 0 -8
	\put{4} at 1 -8
	\put{10} at 2 -8
	\put{24} at 3 -8
	\put{66} at 4 -8

	\put{1} at 0 -10
	\put{2} at 1 -10
	\put{5} at 2 -10

	\put{$$} at -2 2
	\put{$\mathbb E_3 = \mathbb A_2\sqcup\mathbb A_1$} [l] at -3 -0
	\put{$\mathbb E_4 = \mathbb A_4$} [l] at -3 -1
	\put{$\mathbb E_5 = \mathbb D_5$}  [l]  at -3 -2
	\put{$\mathbb E_6$}  [l]  at -3 -3
	\put{$\mathbb E_7$}  [l]  at -3 -4
	\put{$\mathbb E_8$}  [l]  at -3 -5
	\put{$$}  [l]  at -2.2 -6
	\put{$\mathbb B_3$}  [l]  at -3 -7
	\put{$\mathbb F_4$}  [l]  at -3 -8
	\put{$$}  [l]  at -2.2 -9
	\put{$\mathbb G_2$}  [l]  at -3 -10

	\plot -2 1.5  -1.5 1 /
	\put{$\ssize s$} at -1.5 1.5
	\put{$\ssize0$} at 0 1.5
	\put{$\ssize1$} at 1 1.5
	\put{$\ssize2$} at 2 1.5
	\put{$\ssize3$} at 3 1.5
	\put{$\ssize4$} at 4 1.5
	\put{$\ssize5$} at 5 1.5
	\put{$\ssize6$} at 6 1.6
	\put{$\ssize7$} at 7 1.5
	\put{$\ssize8$} at 8 1.5

	\put{} at 0 -10
	\put{sum} [r]  at 11 1.5
	\put{10}  [r] at 11 0
	\put{42}  [r] at 11 -1
	\put{182}  [r] at 11 -2
	\put{833}  [r] at 11 -3
	\put{4160}  [r] at 11 -4
	\put{25080}  [r] at 11 -5
	\put{}  [r] at 11 -6
	\put{20}   [r] at 11 -7
	\put{105}  [r] at 11 -8
	\put{}  [r] at 11 -9
	\put{8}  [r] at 11 -10

	\endpicture}
$$

\section{Hereditary artin algebras}\label{algebras}

\subsection{The basic setting}
Let $\Lambda$ be a hereditary artin algebra. Since by assumption $\Ext_\Lambda^i = 0$ 
for $i \ge 2$, we write
$\Ext(M,M')$ instead of $\Ext_\Lambda^1(M,M')$. The vertices of the
quiver $Q(\Lambda)$ are the isomorphism classes
$[S]$ of the simple $\Lambda$-modules $S$ and there is an arrow $[S] \to [S']$ provided $\Ext(S,S') \neq 0$.
Note that $Q(\Lambda)$ is finite and directed (the latter means that the simple modules can be labeled
$S(i)$ such that the existence of an arrow $[S(i)]\to [S(j)]$ implies that $i > j$).
We endow $Q(\Lambda)$ with a valuation as
follows: given an arrow $[S] \to [S']$, consider $\Ext(S,S')$ as a left $\End(S)^{\text{op}}$-module
and also as a left $\End(S')$-module and put
$$
 v([S],[S']) = (\dim {}_{\End(S)}\Ext(S,S'))(\dim {}_{\End(S')^{\text{op}}}\Ext(S,S'))
 $$
 provided $v([S],[S']) > 1$.
 Given a vertex $i$ of $Q(\Lambda)$, we let $S(i)=S_\Lambda (i), 
P(i) = P_\Lambda(i), I(i)=I_\Lambda(i)$, respectively, denote  a simple,
an indecomposable projective or  injective module
corresponding to the vertex $i$.

If $M$ is a module, the set of 
vertices of the quiver $Q(\Lambda(M))$  
will be called the {\it support} of $M$ and $M$ is said to be {\it sincere} provided
any vertex of $Q(\Lambda)$ belongs to the support of $M$ (thus provided the only idempotent
$e \in \Lambda$ with $eM = 0$ is $e=0$). 

We also will be interested in the corresponding valued graph
$\overline Q(\Lambda)$ which is obtained from the valued quiver $Q(\Lambda)$ by replacing
the arrows by edges: one says that one {\it forgets the orientation} of the quiver.

In the special case where $v([S],[S']) = v$ with $v = 2$ or $v = 3$, it is usual to
replace the arrow
$[S] \longrightarrow [S']$ by a double arrow $[S] \Longrightarrow [S']$ (if $v = 2$) or a similar triple arrow (if $v=3$).
Using the bimodule $\Ext(S,S')$ one obtains an embedding either of $\End S$ into $\End S'$, or of $\End S'$ into
$\End S$; thus one of the division rings is a subring of the other, with index equal to $v$. One marks
the relative size of the endomorphism rings by an additional arrowhead drawn in the middle of the edge, pointing
from the larger endomorphism ring to the smaller one (it should be stressed that these inner arrowheads must not
be confused with the outer ones).
For example, in case there are two simple modules labeled $1$
and $2$ with an arrow $1 \leftarrow 2$ and  $v(1,2) = 2$, there are the following two possibilities:
$$
\hbox{\beginpicture
	\setcoordinatesystem units <1cm,1cm>
	\put{\beginpicture
	\multiput{$\circ$} at  5 0  6 0 /
	\put{$\ssize 1$} at 5 -.3
	\put{$\ssize 2$} at 6 -.3
	\plot 5.8 0.03  5.2 0.03 /
	\plot 5.8 -.03  5.2 -.03 /
	\plot 5.3 0.1  5.1 0  5.3 -.1 /
	\plot 5.6 0.1  5.4 0  5.6 -.1 /
	\endpicture} at 0 0
	\put{\beginpicture
	\multiput{$\circ$} at  5 0  6 0 /
	\put{$\ssize 1$} at 5 -.3
	\put{$\ssize 2$} at 6 -.3
	\plot 5.8 0.03  5.2 0.03 /
	\plot 5.8 -.03  5.2 -.03 /
	\plot 5.3 0.1  5.1 0  5.3 -.1 /
	\plot 5.4 0.1  5.6 0  5.4 -.1 /
	\endpicture} at 3 0

	\endpicture}
	$$
	On the left we see that $\End S(1)$ is a division subring of $\End S(2)$. On the right,
	$\End S(2)$ is a division subring of $\End S(1)$. (Let us exhibit corresponding algebras:
	let $K:k$ be a field extension of degree $2$ and consider the algebras
	$\Lambda = \left[\smallmatrix  k & K \cr
	                                   0 & K \endsmallmatrix\right]$ and
					   $\Lambda' = \left[\smallmatrix K & K \cr
					                                  0 & k \endsmallmatrix\right]$;
  the left quiver shown above is $Q(\Lambda)$, and the right quiver is $Q(\Lambda')$.)

 Here are the corresponding valued graphs, which are
 obtained by forgetting the orientation (thus deleting the outer arrowheads, 
but not the inner ones):
$$
\hbox{\beginpicture
  \setcoordinatesystem units <1cm,1cm>
\put{\beginpicture
  \multiput{$\circ$} at  5 0  6 0 /
  \put{$\ssize 1$} at 5 -.3
  \put{$\ssize 2$} at 6 -.3
  \plot 5.8 0.03  5.2 0.03 /
  \plot 5.8 -.03  5.2 -.03 /
  \plot 5.6 0.1  5.4 0  5.6 -.1 /
\endpicture} at 0 0
\put{\beginpicture
  \multiput{$\circ$} at  5 0  6 0 /
  \put{$\ssize 1$} at 5 -.3
  \put{$\ssize 2$} at 6 -.3
  \plot 5.8 0.03  5.2 0.03 /
  \plot 5.8 -.03  5.2 -.03 /
  \plot 5.4 0.1  5.6 0  5.4 -.1 /
\endpicture} at 3 0
\endpicture}
$$
They are called $\mathbb B_2$ and $\mathbb C_2$, respectively
(observe that there is a difference between $\mathbb B_2$ and $\mathbb C_2$
only if they occur as subgraphs of larger graphs).

     We recall the following \cite{[DR1]}. 
{\it A connected hereditary artin algebra $\Lambda$ is representation-finite if and only if
$\overline Q(\Lambda)$ is one of the Dynkin diagrams 
$$
 \mathbb A_n, \mathbb B_n, \mathbb C_n, \mathbb D_n, \mathbb E_6, \mathbb E_7,
 \mathbb E_8, \mathbb F_4, \mathbb G_2
$$ 
and in this case the indecomposable $\Lambda$-modules correspond bijectively to
the positive roots.}
    
\subsection{Change of orientation}
    
We want to show that the number of basic tilting modules is 
independent of the orientation. 
We recall that a module is said to be {\it basic} provided
it is a direct sum of pairwise non-isomorphic indecomposable modules; an artin algebra
$\Lambda$ is {\it basic} provided the regular representation ${}_\Lambda\Lambda$ is basic.
In case $\Lambda$ is the path algebra of a quiver, we may refer to Ladkani \cite{[L]}. 
In the case
of the tensor algebra of a species (in particular in the case of the path algebra of a quiver),
any change of orientation is obtained by applying a sequence of
BGP-reflection functors;
see \cite{[DR2]}. For a general hereditary artin algebra $\Lambda$, we have to deal with APR-tilting functors as
defined by Auslander, Platzeck and Reiten 
\cite{[APR]}. In order to do so, we may assume that $\Lambda$
is basic. We start with a simple projective module $S$, write ${}_\Lambda\Lambda = S\oplus P$ with a
projective module $P$, and consider $W = P\oplus \tau^-S$ (where $\tau = \tau_\Lambda$ is the
Auslander-Reiten translation in $\mo\Lambda$)
and $\Lambda' = (\End W)^{\text{op}}$. Note that
$W$ is a tilting module (called an APR-tilting module) and the quiver $Q(\Lambda')$ is obtained from the
quiver $Q(\Lambda)$ by changing the orientation of all the arrows which involve the vertex
$\omega = [S]$. 
We let $\Lambda'' = (\End P)^{\text{op}}$ denote the restriction of $\Lambda$ to
the quiver $Q''$ obtained from $Q(\Lambda)$ by deleting the vertex $\omega$ and the arrows
ending in $\omega$. Of course, $Q''$ is also a subquiver of $Q(\Lambda')$ and 
$\Lambda''$ is the restriction of $\Lambda'$ to $Q''$ 
(thus $\Lambda$ is a one-point coextension of $\Lambda''$, whereas
$\Lambda'$ is a one-point extension of $\Lambda''$). We let $S'$ denote the simple 
$\Lambda'$-module with support $\omega$.
		     
\begin{proposition}\label{APR}
 Let $\Lambda$ be a hereditary artin algebra and $S$ a simple projective
 module. Let $W$ be the APR-tilting module defined by $S$ and $\Lambda' = (\End W)^{\text{op}}$.
 Then there is a canonical bijection $\eta$ between the basic tilting $\Lambda$-modules and the basic
 tilting $\Lambda'$-modules.
\end{proposition}

\begin{proof}
In order to define $\eta$, we distinguish two cases.

First, if $T$ is a basic tilting module such that
$S$ is not a direct summand of $T$, let
$\eta(T) = \Hom(W,T)$; this is a basic tilting $\Lambda'$-module and 
$S'$ is not a direct summand of $\eta(T)$.

Second, consider a basic tilting $\Lambda$-module of the form $S\oplus T$. 
Let $T'' = T/U$, where $U$ is the sum of the images of all the maps $S \to T$. Obviously, $T''$ is a basic tilting
$\Lambda''$-module which we may consider as a $\Lambda'$-module.
We form the universal extension $T'$ of $T''$ using copies of $S'$. Then $T'\oplus S'$ is a basic tilting
$\Lambda'$-module (and $S'$ is a direct summand).  
\end{proof}

\begin{remark} 
We may identify the Grothendieck groups $K_0(\Lambda)$ and $K_0(\Lambda')$, using the common factor algebra
$\Lambda''$ and identifying the dimension vectors of $S$ and $S'$. Then, in the first case, the
dimension vector of $\eta(T)$ is obtained from the dimension vector of $T$ by applying the reflection $\sigma$
		  defined by $S$.
		  In the second case, the dimension vectors of $T$ and $\eta(T)$ coincide. Actually, here we use twice the internal
		  reflection defined by $S$ in \cite{[R]}, first in the category $\mo \Lambda$, second in the category $\mo \Lambda'$.
\end{remark}
		  	        
\subsection{The combinatorial backbone}
Let $\Lambda$  be a Dynkin algebra and assume that the vertices
of $Q(\Lambda)$  are labeled $1\le i \le n$.
Let $P(i) = P_\Lambda(i)$ be indecomposable projective.
Since we assume that $\Lambda$
is a Dynkin algebra, there is a natural number $q(i) = q(P(i))$
such that $\tau^{-q(i)}P(i)$ is indecomposable injective;
the modules $M(i,u) = \tau^{-u}P(i)$ with $0 \le u \le q(i)$ and $1\le i \le n$
furnish a complete list of the indecomposable $\Lambda$-modules.
	 
\begin{proposition}\label{hammock}
Let $\Lambda, \Lambda'$ be Dynkin algebras and assume that
the simple modules of both algebras are indexed by $1\le i \le n$.
Assume that $q(P_\Lambda(i)) = q(P_{\Lambda'}(i)) = q(i)$ for all $1\le i \le n$.
If the support of $M(u,i) = \tau_\Lambda^{-u}P_\Lambda(i)$
and $M'(i,u) = \tau_{\Lambda'}^{-u}P_{\Lambda}(i)$ coincide for  all
$0 \le u \le q(i)$ and $1\le i \le n$, then
$a_s(\Lambda) = a_s(\Lambda')$ for all $s$.
\end{proposition}

\begin{proof}
We may interpret the numbers $a_s(\Lambda)$ and
$a_s(\Lambda')$ as the number of antichains in $\mo\Lambda$ and $\mo\Lambda'$,
respectively, which have support-rank $s$. Note that the support of a module $M$
is the set of numbers $1\le i \le n$ such that $\Hom(P(i),M) \neq 0$.

Note that $\Hom(M(i,u),M(j,v)) = 0$ if and only if
$\Hom(M'(i,u),M'(j,v))$ $ = 0$. Namely,
if $u\le v$, the Auslander-Reiten translation (see for example \cite{[ARS]})
furnishes a group isomorphism
\begin{eqnarray*}
 \Hom(M(i,u),M(j,v)) &\simeq& \Hom(M(i,0),M(j,v-u)) \cr
  &=&\Hom(P_{\Lambda}(i),M(j,v-u)),
\end{eqnarray*}
and similarly we have
 $ \Hom(M'(i,u),M'(j,v)) \simeq  \Hom(P_{\Lambda'}(i),M'(j,v-u))$.
 It follows that $\Hom(M(i,u),M(j,v)) = 0$ if and only if
 $i$ is not in the support of $M(j,v-u)$ if and only if
 $i$ is not in the support of $M'(j,v-u)$ if and only if
 $\Hom(M'(i,u),M'(j,v)) = 0$.

 If $u > v$, then
 $$
  \Hom(M(i,0),M(j,v)) \simeq \Hom(M(i,u-v),M(j,0)) = 0,
  $$
  since $M(i,u-v)$ is indecomposable and non-projective, whereas $M(j,0)$ is
  projective. Similarly, we also have
  $ \Hom(M'(i,0),M'(j,v)) = 0$.

  As a consequence we see that given an antichain
  $A = \{A_1,\dots,A_t\}$ in $\mo\Lambda$, the function $M(i,u)\mapsto M'(i,u)$
  yields an antichain $A' = \{A'_1,\dots,A'_t\}$ in $\mo\Lambda'$.
   Of course, the support-rank of $A$ and $A'$ are the same.
   This completes the proof. 
\end{proof}

\begin{corollary}The numbers $a_s(\Lambda)$ depend only on the Dynkin type of $Q(\Lambda)$, not on $\Lambda$ itself.
\end{corollary}

\begin{proof} According to Proposition \ref{hammock}, the numbers $a_s(\Lambda)$ depend only on $Q(\Lambda)$.
According to Proposition \ref{APR}, the orientation of $Q(\Lambda)$ does not play a role.
\end{proof}

Thus, is $\Lambda$ is of Dynkin type $\Delta$, 
we write $a_s(\Delta)$ instead of $a_s(\Lambda)$. 

\begin{corollary}For all $0\le s \le n$, we have $a_s(\mathbb B_n) = a_s(\mathbb C_n)$.
\end{corollary}

\begin{proof}  Apply the Proposition \ref{hammock} to the algebras $\Lambda$ and $\Lambda'$ with valued quivers
$$
\hbox{\beginpicture
	\setcoordinatesystem units <1cm,1cm>
	\put{\beginpicture
	\multiput{$\circ$} at 0 0  1 0  2 0 4 0  5 0 /
	\put{$\cdots$} at 3 0
	\arr{0.8 0}{0.2 0}
	\arr{1.8 0}{1.2 0}
	\arr{2.5 0}{2.2 0}
	\arr{4.8 0}{4.2 0}
	\plot 3.8 0  3.5 0 /
	\put{$\ssize 1$} at 0 -.3
	\put{$\ssize 2$} at 1 -.3
	\put{$\ssize 3$} at 2 -.3
	\put{$\ssize {n-2}$} at 4 -.3
	\put{$\ssize {n-1}$} at 5 -.3
	\put{$\ssize n$} at 6 -.3
	\plot 5.8 0.03  5.2 0.03 /
	\plot 5.8 -.03  5.2 -.03 /
	\plot 5.3 0.1  5.1 0  5.3 -.1 /
	\plot 5.4 0.1  5.6 0  5.4 -.1 /
	\multiput{$\circ$} at 6 0 /
	\endpicture} at 0 0
	\put{\beginpicture
	\multiput{$\circ$} at 0 0  1 0  2 0 4 0  5 0 /
	\put{$\cdots$} at 3 0
	\arr{0.8 0}{0.2 0}
	\arr{1.8 0}{1.2 0}
	\arr{2.5 0}{2.2 0}
	\arr{4.8 0}{4.2 0}
	\plot 3.8 0  3.5 0 /
	\put{$\ssize 1$} at 0 -.3
	\put{$\ssize 2$} at 1 -.3
	\put{$\ssize 3$} at 2 -.3
	\put{$\ssize {n-2}$} at 4 -.3
	\put{$\ssize {n-1}$} at 5 -.3
	\put{$\ssize n$} at 6 -.3
	\plot 5.8 0.03  5.2 0.03 /
	\plot 5.8 -.03  5.2 -.03 /
	\plot 5.3 0.1  5.1 0  5.3 -.1 /
	\plot 5.6 0.1  5.4 0  5.6 -.1 /
	\multiput{$\circ$} at 6 0 /
	\endpicture} at 0 -1
	\endpicture}
$$
respectively; the first valued quiver is of type $\mathbb B_n$, and the second is of
type $\mathbb C_n$. It is well-known (and easy to see) that $q(P_\Lambda(i)) = n\!-\!1
= q(P_{\Lambda'}(i))$ for all $1\le i \le n$ and that the modules
$M(i,u)$ and $M'(i,u)$ for $1\le i\le n$ and $0\le u \le n-1$ have the same
support. \end{proof}

\section{The tilting modules for $\mathbb B_n$}\label{tilting}

We are going to determine the number of tilting modules for the Dynkin algebras of 
type $\mathbb B_n$;
namely we will show that $a_n(\mathbb B_n) = \binom{2n-1}{n-1}$. By induction, we assume knowledge about
the representation theory of $\mathbb B_i$ with $i < n$, as well as the calculation of
$a_s(\mathbb B_n)$ for $s < n$ as shown in Section \ref{hook}.
We consider a Dynkin algebra $\Lambda$ with quiver
$$
\hbox{\beginpicture
	\setcoordinatesystem units <1cm,1cm>
	\multiput{$\circ$} at 0 0  1 0  2 0 4 0  5 0 /
	\put{$\cdots$} at 3 0
	\arr{0.8 0}{0.2 0}
	\arr{1.8 0}{1.2 0}
	\arr{2.5 0}{2.2 0}
	\arr{4.8 0}{4.2 0}
	\plot 3.8 0  3.5 0 /
	\put{$\ssize 1$} at 0 -.3
	\put{$\ssize 2$} at 1 -.3
	\put{$\ssize 3$} at 2 -.3
	\put{$\ssize {n-2}$} at 4 -.3
	\put{$\ssize {n-1}$} at 5 -.3
	\put{$\ssize n$} at 6 -.3
	\plot 5.8 0.03  5.2 0.03 /
	\plot 5.8 -.03  5.2 -.03 /
	\plot 5.3 0.1  5.1 0  5.3 -.1 /
	\plot 5.4 0.1  5.6 0  5.4 -.1 /
	\multiput{$\circ$} at 6 0 /
	\endpicture}
$$
	We interpret $a_n(\mathbb B_n)$ as the number of sincere antichains
(by definition, an antichain $A = \{A_1,\dots, A_t\}$ is {\it sincere} provided
the module $\bigoplus A_i$ is sincere)
and write it as the sum
$$
 a_n(\mathbb B_n) = u(\mathbb B_n) + v(\mathbb B_n)
$$
 where
 $u(\mathbb B_n)$ is the number of antichains with a sincere element,
 whereas $v(\mathbb B_n)$ is the number of sincere antichains without
 a sincere element. These two numbers will be calculated separately.
   
\subsection{The calculation of $u(\mathbb B_n)$}
 We let $w(\mathbb B_n)$ denote the number of antichains which do not contain any
   injective module. 

\begin{lemma}
$$
     w(\mathbb B_n) = a_n(\mathbb B_n).
$$
\end{lemma}

\begin{proof} Let $\mathcal  W$ be the set of antichains without injective modules
     and $\mathcal  S$ the set of sincere antichains.
     We want to construct a bijection $\eta:\mathcal  S \to \mathcal  W$.
     Note that an element of $\mathcal  S$ contains at most one injective module,
     since the injective modules are pairwise comparable with respect to $\Hom$.
     If $A\in \mathcal  S$ contains no injective module, then let $\eta(A) = A$.
     If $A\in \mathcal  S$ contains the injective module $I(i)$, let $\eta(A)$ be obtained
     from $A$ by deleting $I(i)$ and note that $\eta(A)$ is no longer sincere (since
     all the modules $A_j$ in $\eta(A)$ satisfy $\Hom(A_j,I(i)) = 0$). It follows that
     $\eta$ is an injective map. In order to show that $\eta$ is surjective, 
     assume
     that $B$ is an antichain in $\mathcal  W$. If $B$ is sincere, then it belongs to $\mathcal  S$
     and by definition $\eta(B) = B$. If $B$ is not sincere, let $i$ be the smallest
     number such that $i$ is not in the support of $B$. Let $A$ be obtained from $B$
     by adding $I(i)$. Then clearly $A$ is sincere and $\eta(A) = B$. 
\end{proof} 
       
\begin{remark}\label{type-a}
 A similar proof applies to the linearly oriented quiver of type $\mathbb A_n$.
It yields the formula
$$
 a(\mathbb A_{n-1}) = a_n(\mathbb A_n).
$$
Also, instead of looking at antichains which do not contain any injective module, we
may consider antichains which do not contain any projective module. 
\end{remark} 

Now we are able to determine $u(\mathbb B_n)$. 
	
\begin{lemma}
$$
    u(\mathbb B_n) = a_{n-1}(\mathbb B_n) = \binom{2n-2}{n-1}.
$$
\end{lemma}       

\begin{proof}
Note that the sincere indecomposable
representations of $\Lambda$ are the modules
$X(i) = \tau^{-n+i}P(i)$ with $1\le i \le n$. 
The dimension vector
of $X(n)$ is $(1,\dots,1)$, whereas for $1\le i <n$,  the length of $X(n)$
is $n+i$ and its dimension vector is of the form $(1,\dots,1)+(0,\dots,0,1,\dots,1)$.
It is easy to see that $\Hom(X(i),X(j)) \neq 0$ for $i\ge j$, thus any antichain
contains at most one $X(i)$. 
Let $u_{i}(\mathbb B_n)$ be the antichains which contain $X(i)$,
thus
$$
 u(\mathbb B_n) = \sum\nolimits_{i=1}^n u_{i}(\mathbb B_n).
 $$
 Let $\mathcal  X_i$ be the set of indecomposable modules $M$ such that
 $\Hom(X(i),M) = 0 = \Hom(M,X(i))$. Thus, the antichains which contain $X(i)$
 correspond bijectively to the antichains in $\mathcal  X_i$.
 In general, the set $\mathcal  X_i$ consists of three triangles I, II, III:
 $$
 
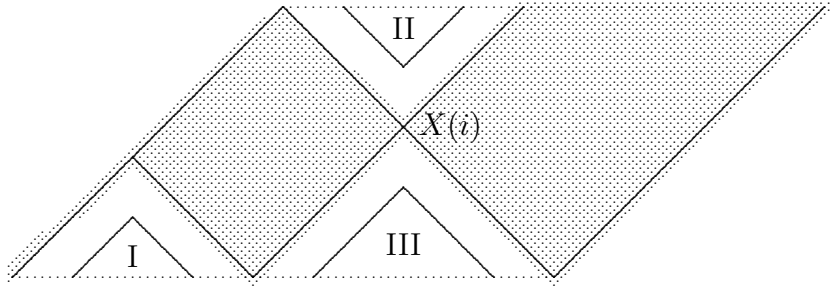
\hbox{\beginpicture
	\setcoordinatesystem units <.4cm,.4cm>
	\multiput{} at 0 0  27 9.5 /
	\plot 0 0  9 9  18 0  27 9 /
	\setdots <1mm>
	\plot  0 0  18 0 /
	\plot  9 9  27 9 /
	\setsolid
	\plot  4 4  8 0  17 9 /
	\plot  2 0  4 2  6 0 /
	\plot  11 9  13 7  15 9 /
	\plot 10 0  13  3  16 0 /
	\setshadegrid span <.5mm>
	\vshade -0.3 -.5 .3 <,z,,> 4 3.7 4.3 <z,z,,> 8 -.3 8.3
	  <z,z,,> 9 0.7 9.2 <z,z,,> 13 4.7 5.3 <z,z,,> 16.8 .7 9.2 <z,z,,> 18 -.3 9.2
	  <z,,,> 27.3 9 9.2 /
	  \put{$X(i)$} [l] at 13.5 5
	  \put{I} at 4 .7
	  \put{II} at 13 8.3
	  \put{III} at 13 1.1
	  \endpicture}
$$
  The triangle I is the wing at the vertex $\tau^{-1}P(n-i-1)$,
  the triangle II is the wing at the vertex $\tau^{-n+i-1}P(i+2)$,
  and the triangle III is the wing at the vertex $\tau^{-n+i+1}P(i-2)$.

  We also are interested in a larger triangle II$'$ which contains the triangle II as well as
  $n-i$ additional modules (all being successors of $X(i)$), namely
  the wing at the
  vertex $\tau^{-n+i}P(i+1)$.

  The full subcategory $\mathcal  X'$ of all direct sums of indecomposable modules 
in the wings I, II$'$, III is the thick subcategory with simple objects
$$
	   S(2), S(3),\dots, S(n-i+1); \quad \tau^{n-i}P(n); \quad  S(n-i+3), \dots, S(n-1).
$$
The position of these modules is indicated here by bullets:
$$
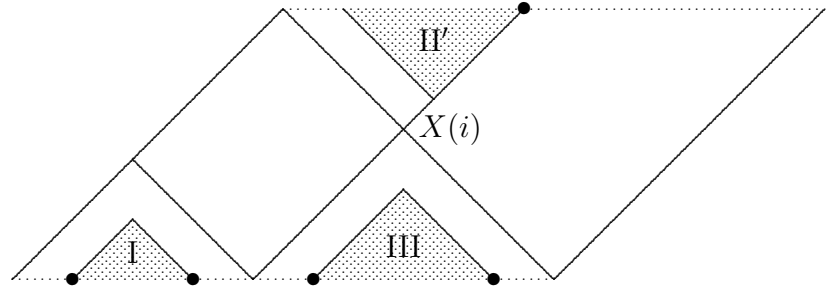
\hbox{\beginpicture
	\setcoordinatesystem units <.4cm,.4cm>
	\multiput{} at 0 0  27 9.5 /
	\plot 0 0  9 9  18 0  27 9 /
	\setdots <1mm>
	\plot  0 0  18 0 /
	\plot  9 9  27 9 /
	\setsolid
	\plot  4 4  8 0  17 9 /
	\plot  2 0  4 2  6 0 /
	\plot  11 9  14 6 /
	\setsolid
	\plot 10 0  13  3  16 0 /
	\setshadegrid span <.5mm>
	\vshade 2  0 0  <,z,,> 4 0 2 <z,,,> 6 0 0 /
	\vshade 11  9 9  <,z,,> 14 6 9 <z,,,> 17 9 9 /
	\vshade 10  0 0  <,z,,> 13 0 3 <z,,,> 16 0 0 /
	\put{$X(i)$} [l] at 13.5 5
	\put{I} at 4 .9
	\put{II$'$} at 14 8
	\put{III} at 13 1.1
	\multiput{$\bullet$} at 2 0   6 0  10 0    16 0  17 9 /
	\endpicture}
$$
(A full subcategory $\mathcal  A$ of $\mo\Lambda$ is called a {\it thick} subcategory 
provided it is
closed under kernels, cokernels and extensions (see for example \cite{[K]}). 
Note that a thick subcategory is an abelian category, and the
inclusion functor $\mathcal  A \to \mo\Lambda$ is exact.)

The category $\mathcal X'$ is of type $\mathbb B_{n-i} \sqcup \mathbb A_{i-2}$
(the $\mathbb A_{i-2}$-part is given by the triangle III, whereas
the $\mathbb B_{n-i}$-part is given by the triangles I and II$'$).
Note that the indecomposables in I and II just correspond to
the non-injective indecomposables in the $\mathbb B_{n-i}$-part.
This shows that
$$
  u_i(\mathbb B_n) = w(\mathbb B_{n-i})a(\mathbb A_{i-2}) 
  = a_{n-i}(\mathbb B_{n-i})a_{i-1}(\mathbb A_{i-1}).
$$

In the special cases $i = 1,2,n-1,n$, the same formula holds. Namely, for $i = 1$ and $i=2$,
the triangle III is empty, whereas the triangles I and II$'$ together yield a category
of type $\mathbb B_{n-i}$.
 In the cases $i=n-1$ and $i=n$, the triangles I and II are empty, whereas the triangle
 III yields a category of type $\mathbb A_{i-2}$.

 Thus we see
$$
	  u(\mathbb B_n) = \sum\nolimits_{i=1}^n u_i(\mathbb B_n) = \sum\nolimits_{i=1}^n
	   a_{i-1}(\mathbb A_{i-1})a_{n-i}(\mathbb B_{n-i}).
$$
	   But the latter expression is the recursion formula for
	   $a_{n-1}(\mathbb B_n)$, since the number of support-tilting modules $T$ with support
	   $\{1,2,\dots,n\}\setminus\{i\}$ is just $a_{i-1}(\mathbb A_{i-1})a_{n-i}(\mathbb B_{n-i})$.
\end{proof} 
	   
\subsection{The calculation of $v(\mathbb B_n)$}

\begin{lemma}
$$
         v(\mathbb B_n) = a_{n-2}(\mathbb B_{n+1}) = \binom{2n-2}{n-2}.
$$
\end{lemma}

\begin{proof} 
Let $\mathcal  V$ be the set
 of sincere antichains of $\Lambda$-modules without
 a sincere element.
 Let $A = (A_1,\dots, A_r)$ be in $\mathcal  V$. Since $A$ is sincere,
we may assume that $\Hom(P(1),A_1)\neq 0$.
Since
$A_1$ is not sincere, we must have
$\Hom(P(n),A_1) = 0$, thus $A_1$ is a representation of a
Dynkin algebra of type $\mathbb A_{n-1}$ and actually an indecomposable
projective representation (also as a $\Lambda$-module), thus
$A_1 = P(i)$ for some $i$ with $1\le i < n$.
Since an antichain can contain only one indecomposable projective module, we see
that $A_1$ is uniquely determined. 

We let $\mathcal  V_i$ denote the sincere antichains $A$ such that $A_1 = P(i)$.
 For $2\le j \le r$, we have
 $\Hom(P(i),A_j) = \Hom(A_1,A_j) = 0$. It follows that
 $(A_2,\dots, A_r)$ is an antichain with support in $[1,i-1]\cup[i+1,n]$.
  Altogether, we see that any element of $A$ has support either in $[1,i]$
  or in $[i+1,n]$.
 The elements of $A$ with support in $[1,i]$ but different from $A_1$
 form  an arbitrary antichain
 with support in $[2,i-1]$, thus the number of elements is $a(\mathbb A_{i-2})$, 
at least  if $i\ge 2$. Note that $a(\mathbb A_{i-2}) = a_{i-1}(\mathbb A_{i-1})$.

 The elements of $A$ with support in $[i+1,n]$ form a sincere antichain
 for $\mathbb B_{n-i}$. Thus the number of such antichains is $a_{n-i}(\mathbb B_{n-i})$.
  This shows that for $i\ge 2$, the set $\mathcal  V_i$ has cardinality
  $a_{i-1}(\mathbb A_{i-1})a_{n-i}(\mathbb B_{n-i})$. This formula holds true also for
  $i=1$, since the number of elements of $\mathcal  V_1$ is  $a_{n-1}(\mathbb B_{n-1})$ and
  $a_0(\mathbb A_0) = 1$. Thus we see that 

\begin{eqnarray*}
 v(\mathbb B_n) &=& \sum\nolimits_{i=1}^{n-1} a_{i-1}(\mathbb A_{i-1})a_{n-i}(\mathbb B_{n-i}) \cr
           &=& -a_{n-1}(\mathbb A_{n-1})
                     +\sum\nolimits_{i=1}^{n} a_{i-1}(\mathbb A_{i-1})a_{n-i}(\mathbb B_{n-i}) \cr
           &=& -\frac1n\binom{2n-2}{n-1} + \binom{2n-2}{n-1}  = \binom{2n-2}{n-2}.
\end{eqnarray*}
Altogether we see
$$
 u(\mathbb B_n) + v(\mathbb B_n) = \binom{2n-2}{n-1} + \binom{2n-2}{n-2}=
 \binom{2n-1}{n-1}.
$$
\end{proof} 	

\begin{remark}
	 The calculation of $v(\mathbb B_n)$ shows the following
		 relationship between the cases $\mathbb A$ and $\mathbb B$:
$$
     a_{n-1}(\mathbb B_n) =   a_{n-2}(\mathbb B_{n+1}) +   a_{n-1}(\mathbb A_{n-1}).
$$
\end{remark}

\section{Support-tilting modules: the hook formula}\label{hook}
		          
\subsection{The hook formula}

\begin{proposition}\label{hook-formula} 
Let $\Delta = \mathbb A, \mathbb B, \mathbb D, \mathbb E$. Then
$$
 a_s(\Delta_n) = a_s(\Delta_{n-1}) + a_{s-1}(\Delta_n)
$$
 for all $n\ge m$ and $1\le s \le n-c$, where $m = 1,2,3,4$ and
 $c=0,1,2,3$ for $\Delta = \mathbb A, \mathbb B,
 \mathbb D, \mathbb E$, respectively.
\end{proposition}

Here we use the convention that $\mathbb B_1 = \mathbb A_1,
\mathbb D_2 = \mathbb A_1\sqcup \mathbb A_1,
\mathbb E_3 = \mathbb A_2\sqcup \mathbb A_1, \mathbb E_4 = \mathbb A_4, \mathbb E_5 = \mathbb D_5$.
In the triangles \ref{triangle1}, \ref{triangle2}, \ref{triangle3}, as well as in 
Section \ref{exceptional}, this equality concerns the following kind
of hooks:
$$
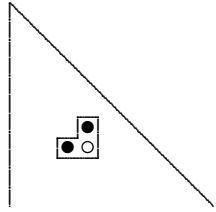
\hbox{\beginpicture
	\setcoordinatesystem units <.9cm,.9cm>
	\multiput{} at 0 0  3 3  /
	\plot 0 0  0 3  3 0 /
	\plot 0.7 0.7  1.3 0.7  1.3 1.3  1 1.3  1 1  0.7 1  0.7 0.7 /
	\multiput{$\bullet$} at 0.85 0.85  1.15 1.15 /
	\put{$\circ$} at 1.15 0.85
	\endpicture}
$$
The hook formula asserts that
the sum of the values at the positions marked by bullets is the value at the
position marked by the circle.
	 
 The various assertions concern the following general situation: up to the choice of
  an orientation, we deal with an artin algebra $\Lambda$ with
  the following valued quiver with $n$ vertices:
$$
	  
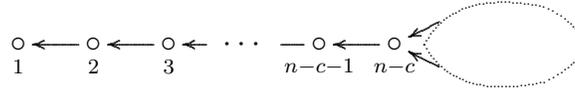
\hbox{\beginpicture
		\setcoordinatesystem units <1cm,1cm>
		\multiput{} at 0 1  8 -1 /
		\multiput{$\circ$} at 0 0  1 0  2 0 4 0  5 0 /
		\put{$\cdots$} at 3 0
		\arr{0.8 0}{0.2 0}
		\arr{1.8 0}{1.2 0}
		\arr{2.5 0}{2.2 0}
		\arr{4.8 0}{4.2 0}
		\plot 3.8 0  3.5 0 /
		\put{$\ssize 1$} at 0 -.3
		\put{$\ssize 2$} at 1 -.3
		\put{$\ssize 3$} at 2 -.3
		\put{$\ssize {n-c-1}$} at 4 -.3
		\put{$\ssize {n-c}$} at 5 -.3
		\arr{5.6 0.3}{5.2 0.1}
		\arr{5.6 -.3}{5.2 -.1}
		\setdots <.5mm>
		\setquadratic
		\plot 5.4 0  6 .5  7 .5  7.5 0  7 -.5  6 -.5  5.4 0 /
		\endpicture}
$$
on the left, we have a quiver of type $\mathbb A_{n-c}$ with arrows $i \leftarrow i\!+\!1$.
The remaining $c$ vertices are in the dotted
``cloud'' to the right. All arrows between the cloud
and the $\mathbb A_{n-c}$--quiver end in the vertex $n-c$. 
We let $Q'$ denote the valued quiver
obtained by deleting the vertex $1$ and the arrow ending in $1$; let $\Lambda'$ be
the corresponding factor algebra of $\Lambda$. Here are the cases we are interested in.
		   
$$
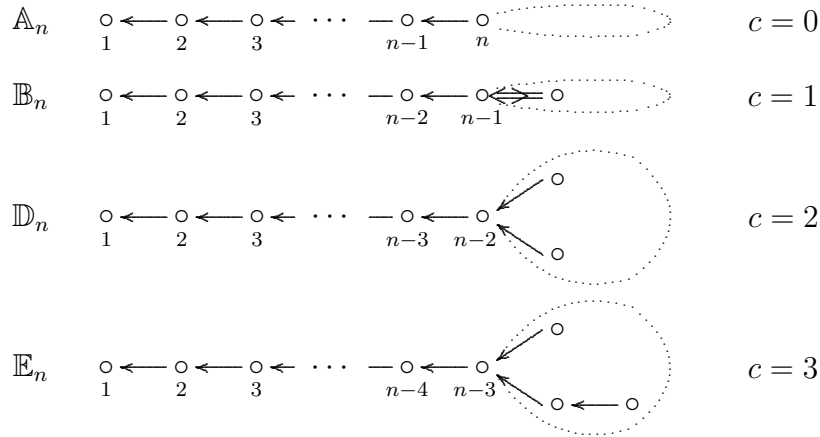
\hbox{\beginpicture
	\setcoordinatesystem units <1cm,1cm>
	\put{\beginpicture
	\multiput{$\circ$} at 0 0  1 0  2 0  4 0  5 0 /
	\put{$\cdots$} at 3 0
	\arr{0.8 0}{0.2 0}
	\arr{1.8 0}{1.2 0}
	\arr{2.5 0}{2.2 0}
	\arr{4.8 0}{4.2 0}
	\plot 3.8 0  3.5 0 /
	\put{$\ssize 1$} at 0 -.3
	\put{$\ssize 2$} at 1 -.3
	\put{$\ssize 3$} at 2 -.3
	\put{$\ssize {n-1}$} at 4 -.3
	\put{$\ssize {n}$} at 5 -.3
	
	\setdots <1mm>
	\setquadratic
	\plot 5.2 0.1  6 .2  7 .2  7.5 0  7 -.2  6 -.2  5.2 -.1 /
	\put{$\mathbb A_n$} at -1 0
	\put{$c = 0$} at 9 0
	\endpicture} at 0 1

	\put{\beginpicture
	\multiput{$\circ$} at 0 0  1 0  2 0 4 0  5 0 /
	\put{$\cdots$} at 3 0
	\arr{0.8 0}{0.2 0}
	\arr{1.8 0}{1.2 0}
	\arr{2.5 0}{2.2 0}
	\arr{4.8 0}{4.2 0}
	\plot 3.8 0  3.5 0 /
	\put{$\ssize 1$} at 0 -.3
	\put{$\ssize 2$} at 1 -.3
	\put{$\ssize 3$} at 2 -.3
	\put{$\ssize {n-2}$} at 4 -.3
	\put{$\ssize {n-1}$} at 5 -.3
	\plot 5.8 0.03  5.2 0.03 /
	\plot 5.8 -.03  5.2 -.03 /
	\plot 5.3 0.1  5.1 0  5.3 -.1 /
	\plot 5.4 0.1  5.6 0  5.4 -.1 /
	\multiput{$\circ$} at 6 0 /
	\setdots <1mm>
	\setquadratic
	\plot 5.2 0.1  6 .2  7 .2  7.5 0  7 -.2  6 -.2  5.2 -.1 /
	\put{$\mathbb B_n$} at -1 0
	\put{$c = 1$} at 9 0

	\endpicture} at 0 0
	\put{\beginpicture
	\multiput{$\circ$} at 0 0  1 0  2 0 4 0  5 0 /
	\put{$\cdots$} at 3 0
	\arr{0.8 0}{0.2 0}
	\arr{1.8 0}{1.2 0}
	\arr{2.5 0}{2.2 0}
	\arr{4.8 0}{4.2 0}
	\plot 3.8 0  3.5 0 /
	\put{$\ssize 1$} at 0 -.3
	\put{$\ssize 2$} at 1 -.3
	\put{$\ssize 3$} at 2 -.3
	\put{$\ssize {n-3}$} at 4 -.3
	\put{$\ssize {n-2}$} at 4.9 -.3
	\arr{5.8 0.5}{5.2 0.1}
	\arr{5.8 -.5}{5.2 -.1}
	\multiput{$\circ$} at 6 .5  6 -.5 /
	\setdots <1mm>
	\setquadratic
	\plot 5.2 0.1  6 .8  7 .8  7.5 0  7 -.8  6 -.8  5.2 -.1 /
	\put{$\mathbb D_n$} at -1 0
	\put{$c = 2$} at 9 0
	\endpicture} at 0 -1.5

	\put{\beginpicture
	\multiput{$\circ$} at 0 0  1 0  2 0 4 0  5 0 /
	\put{$\cdots$} at 3 0
	\arr{0.8 0}{0.2 0}
	\arr{1.8 0}{1.2 0}
	\arr{2.5 0}{2.2 0}
	\arr{4.8 0}{4.2 0}
	\plot 3.8 0  3.5 0 /
	\put{$\ssize 1$} at 0 -.3
	\put{$\ssize 2$} at 1 -.3
	\put{$\ssize 3$} at 2 -.3
	\put{$\ssize {n-4}$} at 4 -.3
	\put{$\ssize {n-3}$} at 4.9 -.3
	\arr{5.8 0.5}{5.2 0.1}
	\arr{5.8 -.5}{5.2 -.1}
	\arr{6.8 -.5}{6.2 -.5}
	\multiput{$\circ$} at 6 .5  6 -.5  7 -.5 /
	\setdots <1mm>
	\setquadratic
	\plot 5.2 0.1  6 .8  7 .8  7.5 0  7 -.8  6 -.8  5.2 -.1 /
	\put{} at -1 0
	\put{$\mathbb E_n$} at -1 0
	\put{$c = 3$} at 9 0

	\endpicture} at 0 -3.5

        \put{} at 0 -4.7
	\endpicture}
$$
		
\begin{lemma} Let $1\le s \le n-c$. Then
$$
 a_s(\Lambda) = a_s(\Lambda') + a_{s-1}(\Lambda).
 $$
 \end{lemma}

\begin{proof}  The support-tilting modules $T$ for $\Lambda$ with $1$ not in the support are just
 the support-tilting modules for $\Lambda'$. 
Let $\mathcal  S_s(\Lambda;1)$ be the set of the basic
support-tilting $\Lambda$-modules $T$ with support-rank $s$ and $\Hom(P(1),T) \neq 0$.
Let $\mathcal  S_{s-1}(\Lambda)$ be the set of basic
support-tilting $\Lambda$-modules $T$ with support-rank $s-1$.
We construct a bijection
$$
  \alpha:\mathcal  S_s(\Lambda;1) \longrightarrow \mathcal  S_{s-1}(\Lambda).
$$
This will establish the formula.

Let $X$ be an indecomposable representation with support-rank $s \le n-c$ and
$\Hom(P(1),X) \neq 0$. Then the support of $X$ is contained in the $\mathbb A_{n-c}$-subquiver,
so $X$ is thin and its support is an interval of the form
$[1,v]$ with $1\le v \le n-c$ (a module is said to be {\it thin}
provided the composition factors are pairwise non-isomorphic; in our setting thin
indecomposable modules are  uniquely determined by the support, thus we may just write
$X = [1,v]$).

Let $T$ be a module in $\mathcal  S_s(\Lambda;1)$.
 At least one of the indecomposable
 direct summand of $T$, say $X$, satisfies 
$\Hom(P(1),X)\neq 0$ and we choose $X = [1,v]$ of largest possible length.
 We claim that $\Hom(P(w),T) = 0$ for any arrow $v \leftarrow w$.
 Assume, to the contrary, that there is an indecomposable direct summand $Y$ of $T$
  with $\Hom(P(w),Y) \neq 0$. 
The maximality of $X$ shows that $\Hom(P(1),Y) = 0$. But then $\Ext(Y,X) \neq 0$
  contradicts the fact that $T$ has no self-extensions. (Namely, if the support of
  $X$ and $Y$ is disjoint, then the arrow $v\leftarrow w$ yields directly a non-trivial
  extension of $X$ by $Y$; if the support of $X$ and $Y$ is not disjoint, then there
  is a proper non-zero factor module of $X$ which is a proper submodule of $Y$,
  thus there is a non-zero map $X \to Y$ which is neither injective nor
  surjective ---
  again we obtain a non-trivial extension of $X$ by $Y$.) Thus the support of $T$
  is the disjoint union of the set $\{1,2,\dots,v\}$ and a set $S''$ which does not
  contain a vertex $w$ with an arrow $v\leftarrow w$.

  The indecomposable direct summands of $T$ with support in
  $\{1,2,\dots,v\}$ yield a tilting module for this $\mathbb A_v$-quiver, and
  $X$ is the indecomposable projective-injective representation of this
  $\mathbb A_v$-quiver. Deleting $X$ from this tilting module, we obtain a support-tilting
  representation of $\mathbb A_v$ with support-rank $v-1$.

  Thus if we write $T = X\oplus T'$, then $T'$ is a support-tilting $\Lambda$-module
  with support-rank $s-1$ (namely, it is the direct sum of a support-tilting module
  with support properly contained in $\{1,2,\dots,v\}$ and a support-tilting
  module with support $S''$). We define $\alpha(T) = T'$; this yields the map
$$
   \alpha:\mathcal  S_s(\Lambda;1) \longrightarrow \mathcal  S_{s-1}(\Lambda)
$$
   we are looking for. It remains to be shown that $\alpha$ is surjective and that we can recover $T$
   from $\alpha(T)$.

   Thus, let $T'$ be in $\mathcal  S_{s-1}(\Lambda)$. Then there are at least $c+1$
   vertices outside of the support of $T'$.

   Case 1: These are the vertices in the cloud
   and precisely one additional vertex, say $i$ (with $1\le i \le n-c$). Note that in this case
   $s = n-c$. Let $T = T'\oplus [1,n-c]$. Since $T'$ is a support-tilting module
   of $\mathbb A_{n-c}$ with support-rank $n-c-1$ and $[1,n-c]$ is the indecomposable
   projective-injective representation of $\mathbb A_{n-c}$, we see that $T = T'\oplus[1,n-c]$
   is a tilting module for $\mathbb A_{n-c}$.

   Case 2: At least two vertices between $1$ and $n-c$ do not belong to $\Supp T'$, say let
   $i<j$ be the smallest such numbers. Then let $T = T'\oplus [1,j-1]$. 
\end{proof}

\subsection{The modified hook formula}

\begin{proposition}\label{modified}
\begin{eqnarray*}
 a_{n-1}(\mathbb D_n) &=& a_{n-1}(\mathbb D_{n-1}) + a_{n-2}(\mathbb D_n) + a_{n-2}(\mathbb A_{n-2}), \\
  a_{n-2}(\mathbb E_n) &=& a_{n-2}(\mathbb E_{n-1}) + a_{n-3}(\mathbb E_n) + a_{n-3}(\mathbb A_{n-3}).
\end{eqnarray*}

\end{proposition}

Again, we consider a general setting, namely
we consider an artin algebra $\Lambda$ with
the following valued quiver with $n$ vertices and we assume that $c \ge 2$:
$$
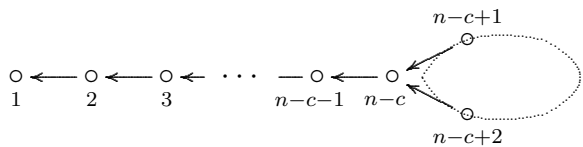
\hbox{\beginpicture
	\setcoordinatesystem units <1cm,1cm>
	\multiput{} at 0 1  8 -1 /
	\multiput{$\circ$} at 0 0  1 0  2 0 4 0  5 0  6 0.5  6 -0.5 /
	\put{$\cdots$} at 3 0
	\arr{0.8 0}{0.2 0}
	\arr{1.8 0}{1.2 0}
	\arr{2.5 0}{2.2 0}
	\arr{4.8 0}{4.2 0}
	\plot 3.8 0  3.5 0 /
	\put{$\ssize 1$} at 0 -.3
	\put{$\ssize 2$} at 1 -.3
	\put{$\ssize 3$} at 2 -.3
	\put{$\ssize {n-c-1}$} at 3.9 -.3
	\put{$\ssize {n-c}$} at 4.9 -.3
	\put{$\ssize {n-c+1}$} at 6 .8
	\put{$\ssize {n-c+2}$} at 6 -.8
	\arr{5.8 0.4}{5.2 0.1}
	\arr{5.8 -.4}{5.2 -.1}
	\setdots <.5mm>
	\setquadratic
	\plot 5.4 0  6 .5  7 .5  7.5 0  7 -.5  6 -.5  5.4 0 /
	\endpicture}
$$
On the left, we have a quiver of type $\mathbb A_{n-c}$ with arrows $i \leftarrow i+1$,
and the remaining $c$ vertices are in the dotted
``cloud'' to the right. There are precisely two vertices in the cloud, namely
$n-c+1$ and $n-c+2$ with arrows $n-c \leftarrow n-c+1$ and
$n-c \leftarrow n-c+2$ and there is no other arrows between the cloud
and the $\mathbb A_{n-c}$ quiver. Again, we let $Q'$ denote the valued quiver
obtained by deleting the vertex $1$ and the arrow ending in $1$ and by $\Lambda'$
	the corresponding factor algebra of $\Lambda$ and we show the following:
	    
\begin{lemma}\label{mod-lemma}
$$
 a_{n-c+1}(\Lambda) = a_{n-c+1}(\Lambda') + a_{n-c}(\Lambda) + a_{n-c}(\mathbb A_{n-c}).
 $$
\end{lemma}

\begin{proof} 
The proof follows closely the proof of Proposition \ref{hook-formula}. 
The support-tilting modules $T$ for $\Lambda$ with $1$ not in the support are just the
support-tilting modules for $\Lambda'$. We construct a surjection
$\alpha$ from the set $\mathcal  S_{n-c+1}(\Lambda;1)$ of the
support-tilting $\Lambda$-modules $T$ with support-rank $n-c+1$ and $\Hom(P(1),T) \neq 0$
onto the set $\mathcal  S_{n-c}(\Lambda)$ of
support-tilting $\Lambda$-modules $T$ with support-rank $n-c$.
In the present setting, $\alpha$ will not be injective, but will be a double cover:
pairs
in  $\mathcal  S(\Lambda;1)$ are identified by $\alpha$; the number of such pairs
will be just $a_{n-c}(\mathbb A_{n-c})$.

As above, one shows that any module $T$ in $\mathcal  S_{n-c+1}(\Lambda;1)$ is of the form 
$T = X\oplus T'$
where $X$ is indecomposable, $\Hom(P(1),X)\neq 0$ and $X$ is of maximal possible length. Note that
the support of $X$ is contained either in $\{1,2,\dots,n-c+1\}$ or in
$\{1,2,\dots,n-c,n-c+2\}$. In particular, $X$ is uniquely determined (since the support
of $T$ cannot contain all the vertices $1,2,\dots,n-c+2$).
As above, the mapping $\alpha$ will be the deletion of the summand $X$.

Let $Z$ be the indecomposable module with
support $\{1,2,\dots,n-c+1\}$ and $Z'$ the indecomposable module with
support $\{1,2,\dots,n-c,n-c+2\}$. Starting with a tilting module $T'$ for
$\mathbb A_{n-c}$, we may form the direct sums $Z\oplus T'$ and $Z'\oplus T'$.
Then these are elements of $\mathcal  S_{n-c+1}(\Lambda;1)$,  both of which are mapped
under $\alpha$ to the same module $T'$. These are the $a_{n-c}(\mathbb A_{n-c})$
pairs of elements of $\mathcal  S(\Lambda;1)$ which are identified by $\alpha$.

It follows that $\mathcal  S(\Lambda;1)$ has cardinality
$a_{n-c}(\Lambda) + a_{n-c}(\mathbb A_{n-c})$.
\end{proof}

\begin{proof}[Proof of Proposition \ref{modified}] The two assertions of 
Proposition \ref{modified} are special cases of Lemma \ref{mod-lemma}.
For the first assertion, $\Lambda$ is of type $\mathbb D_n$, $\Lambda'$ of type
$\mathbb D_{n-1}$, and $c = 2$. 
For the second assertion, $\Lambda$ is of type $\mathbb E_n$, $\Lambda'$ of type
$\mathbb E_{n-1}$ and $c = 3$. 
\end{proof}

\begin{remark} For another proof of the modified hook formula, see Hubery \cite{[H]}.
\end{remark}

\begin{corollary}\label{hook-cor}
$$
 a_{n-1}(\mathbb D_n) = \left[
 \begin{matrix} 2n-3\cr n-1\end{matrix} \right].
$$
 \end{corollary}

\begin{proof}  We start with the previous observation
\begin{eqnarray*}
 a_{n-1}(\mathbb D_n) 
  &=&a_{n-1}(\mathbb D_{n-1}) + a_{n-2}(\mathbb D_n) + a_{n-2}(\mathbb A_{n-2}) \\
    &=& \frac{3n-7}{2n-4}\binom{2n-4}{n-3}
      +\frac{3n-6}{2n-4}\binom{2n-4}{n-2}
        +\frac{1}{n-1}\binom{2n-4}{n-2}.
\end{eqnarray*}
Write
\begin{eqnarray*}
 \binom{2n-4}{n-3} &=& \frac{n-2}{2n-3}\binom{2n-3}{n-1}, \\
  \binom{2n-4}{n-2} &=& \frac{n-1}{2n-3}\binom{2n-3}{n-1}.
\end{eqnarray*}
One easily shows that
$$
  \frac{3n-7}{2n-4}\cdot \frac{n-2}{2n-3} + \frac{3n-6}{2n-4}\cdot\frac{n-1}{2n-3} +
   \frac1{n-1}\cdot\frac{n-1}{2n-3}\ =\ \frac{3n-4}{2n-3}.
$$
   As a consequence, we get
\begin{eqnarray*}
 a_{n-1}(\mathbb D_n) 
     &=&\frac{3n-4}{2n-3}\binom{2n-3}{n-1} \\
     &=& \left[\begin{matrix} 2n-3\cr n-1\end{matrix} \right].
\end{eqnarray*}
\end{proof}

\section{Summation formulas}\label{summation}

An immediate consequence of the previous section is the following assertion:

\begin{proposition} Let $\Delta = \mathbb A$, or $\mathbb B$ and $n\ge 0$, or $\Delta =
\mathbb D$ and $n\ge 2$. If $1\le s \le n-1$, then
$$
  \sum\nolimits_{i=0}^s a_i(\Delta_n) = a_s(\Delta_{n+1}).
$$
\end{proposition}

\begin{proof} We use induction. For $s = 0$ both sides are equal to $1$.
For $s \ge 1$ we have
\begin{eqnarray*}
 \sum\nolimits_{i=0}^s a_i(\Delta_n)
    &=& a_s(\Delta_n) + \sum\nolimits_{i=0}^{s-1} a_i(\Delta_n) \\
    &=& a_s(\Delta_n) + a_{s-1}(\Delta_{n+1}) \\
    &=& a_s(\Delta_{n+1}),
\end{eqnarray*} 
the last equality being the hook formula.
\end{proof}
   
\begin{corollary}Let $\Delta = \mathbb A, \mathbb B$ or $\mathbb D$. Then
$$
        a(\Delta_n) = a_n(\Delta_n)+a_{n-1}(\Delta_{n+1})
$$
\end{corollary}

\noindent
{\bf Case $\mathbb A_n$}
$$
 a(\mathbb A_n) =  \frac 1{n+1}\binom{2n}{n}+
       \frac 3{n+2}\binom{2n}{n-1}
           =   \frac 1{n+2}\binom{2n+2}{n+1}.
$$

\noindent
{\bf Case $\mathbb B_n$}
$$
 a(\mathbb B_n) = \binom{2n-2}{n-1} + \binom{2n-1}n = \binom{2n}n.
$$

\noindent
{\bf Case $\mathbb D_n$}
$$
 a(\mathbb D_n) = \left[\begin{matrix} 2n-2\cr n-2 \end{matrix}\right] +
                   \left[\begin{matrix} 2n-2\cr n-1 \end{matrix}\right]
		                = \left[\begin{matrix} 2n-1\cr n-1 \end{matrix}\right].
$$

\section
    {\bf Acknowledgment} The authors are indebted to Henning Krause and
Dieter Vossieck for providing the references \cite{[GP]} and \cite{[BLR]}, and to
Lutz Hille for helpful discussions concerning the problem of determining the number of tilting modules.
They thank the referee for pointing out mistakes in the proof of Corollary \ref{hook-cor}.
Andrew Hubery has to be praised for his careful reading of the manuscript. His detailed comments have improved the presentation considerably. 

This work is funded by the Deanship of Scientific Research,
King Abdulaziz University, under grant No. 2-130/1434/HiCi.
The authors, therefore, acknowledge technical and financial support of KAU.

\bigskip
\hrule
\bigskip

\noindent 2010 {\it Mathematics Subject Classification}:
Primary:
       05E10. 
				Secondary:
        16G20, 
	16G60,
	05A19, 
	16D90, 
	16G70.

\noindent \emph{Keywords: } 
Dynkin algebra, Dynkin diagram, tilting module, support-tilting module,
     lattice of non-crossing partitions, cluster combinatorics,
generalized Catalan number,
     Catalan triangle, Pascal triangle, Lucas triangle,
          categorification. 

\bigskip
\hrule
\bigskip

\noindent (Concerned with sequences
\seqnum{A009766},
\seqnum{A007318},
\seqnum{A008315}, 
\seqnum{A029635},
\seqnum{A059481},
\seqnum{A129869},
\seqnum{A241188}.)

\end{document}